\documentclass[reqno,english,10pt]{article}

\usepackage{amsmath,amsfonts,amssymb,graphicx,enumerate,url}
\usepackage{stmaryrd}
\usepackage{comment,paralist,mathrsfs}
\usepackage{mathrsfs,booktabs,tabularx}
\usepackage{xifthen,xcolor,tikz,setspace}
\definecolor{internallinkcolor}{rgb}{0,.5,0}
\definecolor{externallinkcolor}{rgb}{0,0,.5}
\usetikzlibrary{decorations.pathmorphing,patterns,shapes,calc,decorations}
\usetikzlibrary{decorations.pathreplacing}
\usepackage{mathtools}
\usepackage[small]{caption}
\usepackage{subcaption}

\usepackage[a4paper]{geometry}
\usepackage{kbordermatrix}

\usepackage[colorinlistoftodos]{todonotes}

\usepackage[colorlinks=true,
  urlcolor=externallinkcolor,
  linkcolor=internallinkcolor,
  filecolor=externallinkcolor,
  citecolor=internallinkcolor,
]{hyperref}

\numberwithin{equation}{section}

\newcommand{\bin}{\operatorname{Bin}}

\renewcommand{\epsilon}{\varepsilon}

\usepackage{color}

\usepackage{amsthm}
\usepackage{thmtools}
\usepackage{bbm}
\usepackage{enumitem}
\usepackage[capitalize]{cleveref}
\Crefname{fact}{Fact}{Facts}
\Crefname{claim}{Claim}{Claims}

\declaretheorem[parent=section]{theorem}
\declaretheorem[sibling=theorem]{lemma}
\declaretheorem[sibling=theorem]{corollary}
\declaretheorem[sibling=theorem]{claim}
\declaretheorem[sibling=theorem]{proposition}

\declaretheorem[sibling=theorem]{conjecture}

\newcommand{\q}{{\bf q}}

\usepackage[nobysame,msc-links,non-sorted-cites]{amsrefs}

\renewcommand{\eprint}[1]{\href{https://arxiv.org/abs/#1}{arXiv:#1}}

\BibSpec{book}{%
	+{}  {\PrintPrimary}                {transition}
	+{,} { \textbf}                     {title} 
	+{.} { }                            {part}
	+{:} { \textit}                     {subtitle}
	+{,} { \PrintEdition}               {edition}
	+{}  { \PrintEditorsB}              {editor}
	+{,} { \PrintTranslatorsC}          {translator}
	+{,} { \PrintContributions}         {contribution}
	+{,} { }                            {series}
	+{,} { \voltext}                    {volume}
	+{,} { }                            {publisher}
	+{,} { }                            {organization}
	+{,} { }                            {address}
	+{,} { \PrintDateB}                 {date}
	+{,} { }                            {status}
	+{}  { \parenthesize}               {language}
	+{}  { \PrintTranslation}           {translation}
	+{;} { \PrintReprint}               {reprint}
	+{.} { }                            {note}
	+{.} {}                             {transition}
	+{}  {\SentenceSpace \PrintReviews} {review}
}
\BibSpec{incollection}{%
  +{}  {\PrintAuthors}                {author}
  +{,} { \textit}                     {title}
  +{.} { }                            {part}
  +{:} { \textit}                     {subtitle}
  +{,} { \PrintContributions}         {contribution}
  +{,} { \PrintConference}            {conference}
  +{}  {\PrintBook}                   {book}
  +{,} { }                            {booktitle}
	+{}  { \PrintEditorsB}              {editor}
	+{,} { }                            {publisher}
  +{,} { \PrintDateB}                 {date}
  +{,} { pp.~}                        {pages}
  +{,} { }                            {status}
  +{,} { \PrintDOI}                   {doi}
  +{,} { available at \eprint}        {eprint}
  +{}  { \parenthesize}               {language}
  +{}  { \PrintTranslation}           {translation}
  +{;} { \PrintReprint}               {reprint}
  +{.} { }                            {note}
  +{.} {}                             {transition}
  +{}  {\SentenceSpace \PrintReviews} {review}
}
\BibSpec{thesis}{%
    +{}  {\PrintAuthors}                {author}
    +{,} { \textit}                     {title}
    +{:} { \textit}                     {subtitle}
    +{,} { \PrintThesisType}            {type}
    +{,} { }                            {organization}
    +{,} { }                            {address}
    +{,} { \PrintDateB}                 {date}
    +{,} { \eprint}                     {eprint}
    +{,} { }                            {status}
    +{}  { \parenthesize}               {language}
    +{}  { \PrintTranslation}           {translation}
    +{;} { \PrintReprint}               {reprint}
    +{.} { }                            {note}
    +{.} {}                             {transition}
    +{}  {\SentenceSpace \PrintReviews} {review}
}
\BibSpec{misc}{%
    +{}  {\PrintAuthors}                {author}
    +{,} { \textit}                     {title}
    +{.} { }                            {part}
    +{:} { \textit}                     {subtitle}
    +{,} { \PrintContributions}         {contribution}
    +{.} { \PrintPartials}              {partial}
    +{,} { }                            {journal}
    +{}  { \textbf}                     {volume}
    +{}  { \PrintDatePV}                {date}
    +{,} { \issuetext}                  {number}
    +{,} { \eprintpages}                {pages}
    +{,} { }                            {status}
    +{,} { \url}                        {url}    
    +{,} { \PrintDOI}                   {doi}
    +{,} { available at \eprint}        {eprint}
    +{}  { \parenthesize}               {language}
    +{}  { \PrintTranslation}           {translation}
    +{;} { \PrintReprint}               {reprint}
    +{.} { }                            {note}
    +{.} {}                             {transition}
    +{}  {\SentenceSpace \PrintReviews} {review}
}

\newcommand{\cA}{\mathcal A}
\newcommand{\cB}{\mathcal B}

\newcommand{\cG}{\mathcal G}
\newcommand{\cM}{\mathcal M}

\newcommand{\bv}{\mathbf v}

\newcommand{\p}{{\bf p}}
\newcommand{\na}{\mathsf{a}}

\DeclareMathOperator{\rank}{rank}

\newcommand{\defn}[1]{{\bfseries #1}}

\newcommand{\eps}{\varepsilon}
\renewcommand{\phi}{\varphi}
\newcommand{\RR}{\mathbb{R}}
\renewcommand{\SS}{\mathbb{S}}

\newcommand{\sm}{\smallsetminus}
\newcommand{\es}{\varnothing}

\newcommand{\vect}{\mathbf}
\newcommand{\trans}{^\mathsf{T}}

\newcommand{\pr}{\mathbb{P}}
\newcommand{\E}{\mathbb{E}}
\newcommand{\whp}[0]{\textbf{whp}}
\newcommand{\Dist}[1]{\mathsf{#1}}
\newcommand{\Bin}{\Dist{Bin}}

\newcommand{\sL}{\mathsf{L}}
\newcommand{\R}{\mathsf{R}}
\newcommand{\Rhat}{\hat{\mathsf{R}}}

\newcommand{\drat}{\Gamma}
\newcommand{\nlb}{\alpha} 
\newcommand{\psb}{\zeta} 

\begin{document}

\title{Rigid partitions: from high connectivity to random graphs}

\author{
  Michael Krivelevich\thanks{
    School of Mathematical Sciences,
    Tel Aviv University,
    Tel Aviv 6997801, Israel.
    Email: \href{mailto:krivelev@tauex.tau.ac.il}
                {\tt krivelev@tauex.tau.ac.il}.
  }
  \and
  Alan Lew\thanks{
    Department of Mathematical Sciences,
    Mellon College of Science,
    Carnegie Mellon University,
    Pittsburgh, PA 15213, USA.
    Email: \href{mailto:alanlew@andrew.cmu.edu}
                {\tt alanlew@andrew.cmu.edu}.
  }
  \and
  Peleg Michaeli\thanks{
    Mathematical Institute,
    University of Oxford,
    Oxford, UK.
    Email: \href{mailto:peleg.michaeli@maths.ox.ac.uk}
                {\tt peleg.michaeli@maths.ox.ac.uk}.
    Research partially supported by ERC Advanced Grant 883810.
  }
}

\date{}

\maketitle

\begin{abstract}
  A graph is called $d$-rigid if there exists a generic embedding of its vertex set into $\mathbb{R}^d$
  such that every continuous motion of the vertices that preserves the lengths of all edges
  actually preserves the distances between all pairs of vertices.
  The rigidity of a graph is the maximal $d$ such that the graph is $d$-rigid.
  We present new sufficient conditions for the $d$-rigidity of a graph
  in terms of the existence of ``rigid partitions'' ---
  partitions of the graph that satisfy certain connectivity properties.
  This extends previous results by Crapo,
  Lindemann,
  and Lew, Nevo, Peled and Raz.

  As an application, we present new results on the rigidity of
  highly-connected graphs, random graphs, random bipartite graphs,
  pseudorandom graphs, and dense graphs.
 In particular, we prove that random $C d\log d$-regular graphs are typically $d$-rigid, demonstrate the existence of a giant $d$-rigid component in sparse random binomial graphs, and show that the rigidity of relatively sparse random binomial bipartite graphs is roughly the same as that of the complete bipartite graph, which we consider an interesting phenomenon. Furthermore, we show that a graph admitting $\binom{d+1}{2}$ disjoint connected dominating sets is $d$-rigid. This implies a weak version of the Lov\'asz--Yemini conjecture on the rigidity of highly-connected graphs.
  We also present
  an alternative short proof for a recent result by Lew, Nevo, Peled, and Raz,
  which asserts that the hitting time for $d$-rigidity in the random graph process
  typically coincides with the hitting time for minimum degree $d$.
\end{abstract}

\section{Introduction}
A \defn{$d$-dimensional framework} is a pair $(G,\p)$ where $G=(V,E)$ is a finite simple graph
and $\p:V\to\RR^d$ is an embedding of its vertices.
A framework is called \defn{rigid} if every continuous motion of the vertices
that preserves the lengths of all the edges does not change the distance between any two vertices.
An embedding of $G=(V,E)$ is \defn{generic} if its $d|V|$ coordinates are
algebraically independent over the rationals.
A framework with a generic embedding will be called a \defn{generic} framework.

We define the \defn{rigidity matrix} $\R(G,\p)\in\RR^{|E|\times d|V|}$
of a framework $(G,\p)$ as follows:
the rows of $\R(G,\p)$ are indexed by the edge set $E$
and the columns are indexed by allocating $d$ coordinates to each vertex of $V$.
The row vector $\vect{r}_e$ indexed by $e=\{u,v\}$ is supported on the coordinates of $u,v$
where it is equal to the $d$-dimensional row vectors $\p(u)-\p(v)$ and $\p(v)-\p(u)$,
respectively. For simplicity, we assume that $G$ has at least $d+1$ vertices.
It is known (see~\cite{AR78}) that the rank of $\R(G,\p)$ is always at most $d|V|-\binom{d+1}{2}$.  The framework $(G,\p)$  is called \defn{infinitesimally rigid} if the rank of $\R(G,\p)$
is exactly $d|V|-\binom{d+1}{2}$.
Infinitesimal rigidity is in general a stronger condition than rigidity;
however, it was shown by Asimow and Roth in~\cite{AR78} that both conditions are equivalent for generic embeddings. 

We say that a graph $G=(V,E)$ is \defn{rigid in $\RR^d$} (or, \defn{$d$-rigid}, for short)
if there exists an embedding $\p:V\to\RR^d$ such that the framework $(G,\p)$ is infinitesimally rigid (or equivalently, if $(G,\p)$ is rigid for all generic embeddings $\p:V\to\RR^d$;
see~\cite{AR78}). We define the \defn{rigidity} of $G$ to be the maximal $d$ such that $G$ is $d$-rigid.
$d$-rigidity is thus a graph-theoretic property;
however, a purely combinatorial characterization remains a major open problem in rigidity theory for $d>2$
(see, e.g.,~\cite{schulze2017rigidity}).

Rigidity theory has a long history, tracing its origins to early results such as 
Cauchy's theorem (see, e.g.,~\cite{AZ}),
and has found applications (especially in two and three dimensions) in various 
fields outside of mathematics (see, e.g., the survey~\cite{SJ17}).
In particular, significant efforts have been directed towards characterising 
$d$-rigid graphs, with seminal works~\cite{PG27}, \cite{Lam70}, and~\cites{AR78,AR79}.
For further reading, we refer the reader to the textbook~\cite{GSS} and the 
survey~\cite{Con93}.

In this paper,
we present a new sufficient condition for $d$-rigidity based on graph partitions.
We apply this to show that high connectivity implies high rigidity,
a first result of its kind for dimensions greater than $2$.
We also utilize it to establish lower bounds for the rigidity of (pseudo)random and dense graphs,
derive a simplified proof for the lower bound of rigidity in complete bipartite graphs,
and examine a notable phenomenon in random bipartite graphs.

\subsection{Rigid partitions}
Let $G=(V,E)$ be a graph.
For pairwise disjoint vertex sets $V_1,\dots,V_k$,
we say that $V_1,\dots,V_k$
is a \defn{partial colouring} of $V$ if $V_1\cup\dots\cup V_k\subseteq V$,
or a \defn{colouring} of $V$ if $V_1\cup\dots\cup V_k=V$.
A \defn{partition} of $V$ is a colouring $V_1,\dots,V_k$ of $V$
in which $V_i\ne\es$ for all $1\le i\le k$.
For two (not necessarily disjoint) vertex sets $A,B\subseteq V$,
we let $E(A, B)$
be the set of edges having one endpoint in $A$ and the other in $B$.
Let $E_1,\dots,E_k$ be a partial colouring of the edge set of $G$,
and let $\hat{E}=E_1\cup\dots\cup E_k$.
We say that $U\subseteq V$ has a \defn{monochromatic cut}
(with respect to the partial colouring)
if there exists a partition $U=U'\cup U''$
such that $\hat{E}\cap E(U',U'')\subseteq E_i$ for some $i\in[k]$.

Now, fix $d\ge 1$,
let $V_1\cup \dots \cup V_{d+1}$ be a colouring of $V$,
and let $\{E_{ij}\}_{1\le i<j\le d+1}$ be a partial colouring of $E$,
for which $e\subseteq V_i\cup V_j$ for all $1\le i<j\le d+1$ and $e\in E_{ij}$.
Suppose the subgraphs $G_{ij}=(V_i\cup V_j,E_{ij})$ are connected
for all $1\le i<j\le d+1$ (where we consider a graph without vertices to be disconnected),
and for every $1\le i\le d+1$ and $U\subseteq V_i$ with $|U|\ge 2$,
$U$ has a monochromatic cut.
Then, we call the pair $\left(\{V_i\}_{i=1}^{d+1}, \{E_{ij}\}_{1\leq i<j\leq d+1}\right)$
a \defn{$d$-rigid partition} of $G$.
For notational convenience, a vertex partition $\{V_i\}_{i=1}^{d+1}$
is also called a $d$-rigid partition
if there exists an associated partial colouring $\{E_{ij}\}_{1\le i<j\le d+1}$
for which the aforementioned pair
is a $d$-rigid partition. Note that in a $d$-rigid partition of $G$ one of the vertex sets $V_i$ may be empty.

Our main theorem states that
any graph that admits a \(d\)-rigid partition is also \(d\)-rigid.

\begin{theorem}\label{thm:rigid:partition}
  A graph that admits a $d$-rigid partition is $d$-rigid.
\end{theorem}

The case $d=2$ of \cref{thm:rigid:partition} was addressed by Crapo~\cite{Cra90},
who proved that a graph is $2$-rigid if and only if it admits a $2$-rigid partition (in fact, Crapo used a different but equivalent notion of ``proper $3T2$ decompositions").
For $d=3$,
Lindemann provided a sufficient condition for $3$-rigidity in his PhD thesis~\cite{LinPhD}. His condition is described in different terms, using the notion of ``proper $d$-framework decompositions'' (which extends Crapo's notion of $3T2$ decompositions in the $d=2$ case), but it can be shown to be equivalent to the condition stated in \cref{thm:rigid:partition}.
Lindemann conjectured (using a different terminology; see \cite{LinPhD}*{Conjecture~5.6}) that the existence of a $3$-rigid partition characterizes  $3$-rigid graphs, and asked whether the same holds for general $d$. We show here that, for some values of $d$ (and assuming the validity of certain conjectures on the rigidity of specific families of graphs, for all large enough $d$), the existence of a $d$-rigid partition is \emph{not} a necessary condition for $d$-rigidity
(see discussion in \cref{sec:remarks}).

Recently,
Jord\'an and Tanigawa~\cite{JT22} introduced the \defn{$d$-dimensional algebraic connectivity}
$\na_d(G)$ of a graph $G$,
which is a high dimensional generalisation of the algebraic connectivity of $G$
introduced by Fiedler in~\cite{Fie73} (corresponding to the $d=1$ case),
and can be seen as a quantitative measure of the $d$-rigidity of $G$
(see \cref{sec:stiffness} for a precise definition).
Using this notion, we establish here a ``quantitative'' version of \cref{thm:rigid:partition}
(see \cref{thm:rigid:partition:quant}),
which directly implies \cref{thm:rigid:partition}.

In practice, it is often difficult to apply \cref{thm:rigid:partition} directly.
Therefore, we introduce a few concrete graph partitions that are more applicable.
Let $G=(V,E)$ be a graph.
A \defn{dominating set} for $G$ is a subset $S$ of its vertices,
such that any vertex of $G$ is either in $S$,
or has a neighbour in $S$. We say that a vertex set $S$ is \defn{connected} if the induced graph $G[S]$ is connected.
A \defn{connected dominating set} (CDS) is a dominating set that is also connected.
A \defn{CDS partition of size $k$} is a partition $V_1,\dots,V_k$ of $V$
such that each $V_i$ is a CDS. Note that, since a set containing a CDS is also a CDS, the existence of $k$ vertex disjoint connected dominating sets implies the existence of a CDS partition of size $k$.
Our first application of \cref{thm:rigid:partition} is the following.
\begin{corollary}\label{cor:cdsp:rigid}
  A graph that admits a CDS partition of size $\binom{d+1}{2}$ is $d$-rigid.
\end{corollary}
A proof of \cref{cor:cdsp:rigid} appears in \cref{sec:rigid:partitions:special}.
We use \cref{cor:cdsp:rigid} to prove that highly-connected graphs are highly-rigid.
Let us make this statement precise.
It is well known (see e.g. \cite{JJ05}) that $d$-rigid graphs are $d$-connected.
Lov\'asz and Yemini~\cite{LY82} conjectured in 1982 that there exists a (smallest) integer $c_d$,
possibly $c_d=d(d+1)$, such that every $c_d$--connected graph is $d$-rigid.
The $d=2$ case was solved by Lov\'asz and Yemini in \cite{LY82}.
Some progress has been made in recent years in the $d>2$ case
(see e.g. \cites{Jor12,JT22pow} for related results on powers of graphs, and \cite{JJ05} for a partial result in the case $d=3$). In particular, Clinch, Jackson and Tanigawa proved in \cite{clinch2022abstract} an analogous version of the conjecture for the ``$C_2^1$-cofactor matroid", which is conjectured to be equal to the generic $3$-rigidity matroid
(\cite{whiteley1995some}*{Conjecture~10.3.2}). As an application of \cref{cor:cdsp:rigid}, we prove the following.

\begin{theorem}[Highly connected graphs]\label{thm:connectivity}
  There exists $C>0$ such that for all $d=d(n)\geq 1$,
  every $n$-vertex $Cd^2{\log^2}{n}$-connected graph is $d$-rigid.
\end{theorem}

The proof of \cref{thm:connectivity} follows by combining \cref{cor:cdsp:rigid} with a result
of Censor-Hillel, Ghaffari, Giakkoupis, Haeupler and Kuhn~\cite{CGGHK17}
that states that every $k$-connected graph has a CDS partition of size $\Omega(k/{\log^2}{n})$;
see \cref{sec:conn}.
We wish to highlight that shortly after making our manuscript available online,
the full resolution of the Lov\'asz--Yemini conjecture
was announced by Vill\'anyi~\cite{villanyi2023+every}.
Since every graph admitting a CDS partition of size $k$
is either complete or $k$-connected (as shown in~\cite{Zel86}),
Vill\'anyi's results imply a weaker version of \cref{cor:cdsp:rigid},
requiring a partition of size $d(d+1)$
rather than $\binom{d+1}{2}$.

We proceed by presenting a few additional useful rigid partition.
For two (not necessarily disjoint) vertex sets $A,B\subseteq V$,
write $G[A,B]$ for the subgraph of $G$ on the vertex set $A\cup B$
with the edge set $E(A,B)$.
We say that a partition $V_1,\dots,V_d$ of $V$
is a \defn{(type I) strong $d$-rigid partition} of $G$,
and that $G$ {\em admits} a (type I) strong $d$-rigid partition,
if $G[V_i,V_j]$ is connected for all $1\leq i\leq j\leq d$.
Throughout this paper we will mostly work with this definition,
and omit the ``type I'' prefix.
We say that a partition $V_1,\dots,V_{d+1}$ of $V$
is a \defn{type II strong $d$-rigid partition} of $G$,
and that $G$ {\em admits} a (type II) strong $d$-rigid partition,
if $G[V_i,V_j]$ is connected for all $1\le i<j\le d+1$.
For strong $d$-rigid partitions of type II, unlike type I, there is no connectivity requirement for the induced subgraphs $G[V_i,V_i]=G[V_i]$, but a partition into $d+1$ parts is necessary, as opposed to the $d$ parts required in type I. It is important to note that neither definition implies the other; however, both types qualify as rigid partitions, as detailed in \cref{sec:rigid:partitions:special}.

\begin{corollary}\label{cor:strong:rigid:partition}
  A graph that admits a (type I or type II) strong $d$-rigid partition is $d$-rigid.
\end{corollary}

A quantitative version of the ``type I'' case in \cref{cor:strong:rigid:partition}
has recently been proved by the second author with Nevo, Peled and Raz~\cite{LNPR23+}*{Theorem~1.3}.
Indeed, our proof of \cref{thm:rigid:partition:quant}
relies on an extension of their arguments. 

We would like to highlight a few limitations of proving $d$-rigidity of a graph
by showing that it admits a strong $d$-rigid partition.
The \defn{domination number} of $G$, denoted by $\gamma(G)$,
is the number of vertices in a smallest dominating set for $G$.
Observe that if an $n$-vertex graph $G$ admits a strong $d$-rigid partition
then $\gamma(G)\le n/d$.
Thus, this method can only be used to prove that $G$ is $d$-rigid
for $d\le n/\gamma(G)$.
A second limitation is that bipartite graphs do not admit strong $d$-rigid partitions for $d\ge 3$.
Indeed, let $G=(A\cup B,E)$ be a bipartite graph with bipartition $(A,B)$,
and let $V_1,\dots,V_d$ be a partition of $V$ for $d\ge 3$.
If there exists $1\le i<j \le d$ such that $V_i,V_j$ are both in $A$ or in $B$
then $G[V_i,V_j]$ is not connected, hence this is not a strong $d$-rigid partition.
Otherwise, there exists $1\le i\le d$ such that $V_i$ intersects both $A$ and $B$.
Thus, for $j\ne i$, $G[V_i,V_j]$ is not connected,
and again, $V_1,\dots,V_d$ is not a strong $d$-rigid partition.
Nevertheless, bipartite graphs may be $d$-rigid for arbitrarily large $d$.
This motivates the following bipartite variant of the definition above.

Let $G=(A\cup B,E)$ be a bipartite graph with bipartition $(A,B)$. Let $V_1,\ldots, V_{d+1}$ be a partition of $V=A\cup B$. Write $A_i=V_i\cap A$ and and $B_i=V_i\cap B$ for $1\le i\le d+1$.
We say that the partition is a \defn{strong bipartite $d$-rigid partition} of $G$,
and that $G$ {\em admits} a strong bipartite $d$-rigid partition,
if the following conditions hold:
$G[A_i,B_j]$ is connected for all $i\ne j$,
and there exists a sequence $0\leq s_1\leq s_2\leq \cdots\leq s_{d+1}$
satisfying $\sum_{i=1}^k s_i\geq \binom{k}{2}$ for all $1\leq k\leq d+1$,
such that for every $1\leq i\leq d+1$, $G[A_i,B_i]$ contains a forest with $s_i$ edges.

\begin{corollary}\label{cor:strong:bipartite:rigid:partition}
  A bipartite graph that admits a strong bipartite $d$-rigid partition is $d$-rigid.
\end{corollary}

Whiteley~\cite{Whi84}
(see also~\cite{Jor22}*{Theorem~1.2} and \cite{Nix21}*{Corollary~4}),
and independently Raymond~\cite{Ray84},
both relying on previous work by Bolker and Roth~\cite{BR80}, showed that,
for $m,n\geq 2$ and $d\geq 1$,
the complete bipartite graph $K_{m,n}$ is $d$-rigid
if and only if $m,n\geq d+1$ and $m+n\geq \binom{d+2}{2}$.
We show here that, in fact,
any complete bipartite graph $K_{m,n}$ satisfying these conditions
admits a strong bipartite $d$-rigid partition.

\begin{proposition}\label{prop:complete:bipartite}
  Let $d\ge 1$.
  If $m,n\geq d+1$ and $m+n\geq \binom{d+2}{2}$,
  then $K_{m,n}$ admits a strong bipartite $d$-rigid partition.
\end{proposition}

Combining \cref{cor:strong:bipartite:rigid:partition,prop:complete:bipartite},
we obtain a new simple proof of the sufficient condition for the $d$-rigidity of complete bipartite graphs. 
The proofs appear in \cref{sec:rigid:partitions:special,sec:bipartite},
respectively.

\subsection{Applications to random and pseudorandom graphs}
We present several applications of
\cref{cor:strong:rigid:partition,cor:strong:bipartite:rigid:partition}
to random and pseudorandom graphs.
For all instances where we apply \cref{cor:strong:rigid:partition},
we demonstrate that a particular graph (typically) admits a (type I) strong $d$-rigid partition.
In these scenarios, similar reasoning,
which we forego detailing,
can be used to show that the graph also admits a type II strong $d$-rigid partition.

Our first application is for pseudorandom graphs,
for which we are going to use the following linear-algebraic definition. Let $\lambda_1(G)\ge \lambda_2(G)\ge \cdots \geq \lambda_n(G)$ be the eigenvalues of the adjacency matrix of a graph $G$. 
An $n$-vertex $k$-regular graph $G$ is an \defn{$(n, k, \lambda)$-graph} if $\max\{|\lambda_2(G)|,|\lambda_n(G)|\}\leq \lambda$.
 For a comprehensive review on pseudorandom graphs, see~\cite{KS06}.
\begin{theorem}[Pseudorandom graphs]\label{thm:nkl}
  There exists $C>1$ such that if
  $G$ is an $(n,k,\lambda)$-graph with
  $k\ge\max\{9d\lambda,Cd\log{d}\}$
  then $G$
  admits a strong $d$-rigid partition.
\end{theorem}
We remark that here and later we prefer
clarity of presentation over optimisation of the constants.
\cref{thm:nkl} will follow immediately
from a more general statement (\cref{prop:jumbled:is:rigid})
about ``jumbled'' graphs; see \cref{sec:jumbled}.
Note that in light of Vill\'anyi's recent findings \cite{villanyi2023+every},
one can establish
a stronger rigidity bound for graphs with mild expansion.
Given that every $k$-regular graph is $(k-\lambda_2(G))$-connected (see \cite{Fie73}),
it follows that such a graph is $(1-o(1))\sqrt{k-\lambda_2(G)}$-rigid.
See also \cites{CDG21,FHL23} for some related results in the $d=2$ case.

Since random regular graphs are, \whp{}{}\footnote{%
  With high probability, i.e., with probability tending to $1$ as $n$ tends to $\infty$.}
$(n,k,\lambda)$-graphs for $\lambda=\Theta(\sqrt{k})$ (see, e.g.,~\cites{FKS89,Fri08,Sar23}),
it follows immediately from \cref{thm:nkl} that a random $\Omega(d^2)$-regular graph is \whp{} $d$-rigid.
Using a more direct argument, we are able to prove the following stronger statement.
Let $\cG_{n,k}$ denote the distribution of random $k$-regular graphs. We denote a graph sampled from this distribution as $G\sim \cG_{n,k}$, or occasionally, slightly abusing notation, simply as $\cG_{n,k}$.
\begin{theorem}[Random regular graphs]\label{thm:gnr}
  There exists $C>0$ such that
  for every fixed $d\geq 2$ and $k\ge C d \log{d}$, $G\sim \cG_{n,k}$ \whp{} admits a strong $d$-rigid partition.
\end{theorem}

Previously, the sole result regarding the rigidity of random regular graphs 
came from~\cite{JSS07}, demonstrating that \( \cG_{n,k} \) is \whp{} \( 2 \)-rigid 
when \( k\ge 4 \).
It is plausible, however, that \( \cG_{n,2d} \) is
\whp{} \( d \)-rigid for every \( d\ge 2 \).
Our result can be viewed as a first step in that direction.

Consider the \defn{binomial random graph} $G(n,p)$
which is the distribution over graphs on the vertex set $[n]$
in which each pair of vertices is connected independently with probability $p$.
It is well known~\cite{BT87} that every (nontrivial) monotone graph property has a threshold
in $G(n,p)$.
Since rigidity is evidently a monotone property, it is natural to try and identify its threshold.
The initial efforts in this direction were made by Jackson, Servatius, and Servatius~\cite{JSS07}
who proved that $p\ge(\log{n}+\log\log{n}+\omega(1))/n$ suffices to guarantee $2$-rigidity \whp{}.
As for $d>2$, further advancements were achieved by Kir\'aly and Theran~\cite{KT13+},
as well as by Jordán and Tanigawa~\cite{JT22}.
Eventually, Lew, Nevo, Peled, and Raz~\cite{LNPR23} established that
$p_c(d)=(\log{n}+(d-1)\log\log{n})/n$ is the threshold for
$d$-rigidity for any fixed value of $d$. 
In fact, for fixed $d$, they also proved a more robust {\em hitting-time result},
which we discuss next.

The \defn{random graph process}
on the vertex set $[n]$
is a stochastic process that starts with the vertex set $[n]$ and no edges,
and at each step, a new edge is added by selecting uniformly from the set of missing edges.
We denote by $G(n,m)$ the random graph process on $n$ vertices at time $m$;
note that this is a uniformly chosen random graph on $n$ vertices with $m$ edges.
The \defn{hitting time} of a monotone property refers to the first time at which
the random graph process possesses that particular property.
An advantage of studying random graph processes is that
it allows for a higher resolution analysis of the typical emergence of (monotone) graph properties.
The main result of~\cite{LNPR23} is that for every fixed $d\ge 1$,
the hitting time for minimum degree $d$
coincides, \whp{},
with the hitting time for $d$-rigidity\footnote{In fact, the results from~\cite{LNPR23} hold also for more general notions of rigidity,
defined in terms of ``abstract rigidity matroids'' (see e.g.~\cite{Gra91}).}.
Here, we extend this result for sufficiently slowly growing $d$,
and strengthen it by showing that the hitting time for minimum degree $d$
coincides, \whp{}, with the hitting time for the admittance of strong $d$-rigid partitions.
Let $\tau_d$ denote the hitting time for minimum degree $d$
in a random graph process.

\begin{theorem}[Hitting time for $d$-rigidity]\label{thm:gnt}
  There exists $c>0$ for which the following holds.
  For every $1\le d=d(n)\le c\log{n}/\log\log{n}$,
  a random graph
  $G\sim G(n,\tau_d)$ \whp{} admits a strong $d$-rigid partition.
\end{theorem}
The above theorem is sharp in the sense that $d$-rigidity implies $d$-connectivity,
and hence also minimum degree $d$.
For $d\gg \log{n}/\log\log{n}$, however, the threshold for $d$-rigidity remains unknown.
In this work, we establish that $G(n,p)$ is \whp{} $d$-rigid for $d=\Omega(np/\log(np))$,
provided that $p\ge (1+\eps)\log{n}/n$,
which is best possible up to logarithmic factors.
Moreover, we show the existence of a strong $d$-rigid partition,
an interesting result in its own right.

\begin{theorem}[Binomial random graphs]\label{thm:gnp}
  For every $\eps>0$ there exists $c>0$ such that if $p\ge (1+\eps)\log{n}/n$, then  $G\sim G(n,p)$ \whp{} admits a strong $d$-rigid partition for $d=cnp/\log(np)$.
\end{theorem}

We emphasize that the above statement is sharp, up to constants,
in the following sense:
a graph that admits a strong $d$-rigid partition has, in particular, a dominating set of size $n/d$.
The smallest dominating set in $G(n,p)$ is \whp{} of size $\Omega(\log(np)/p)$
(see, e.g.,~\cite{GLS15}),
implying that the largest $d$ for which $G(n,p)$ admits a strong $d$-rigid partition
is \whp{} $O(np/\log(np))$.
Nevertheless, we conjecture that when $f\to\infty$,
the ``bottleneck'' for the rigidity of $G(n,p)$ lies in its number of edges: note that a $d$-rigid $n$-vertex graph must have at least $dn-\binom{d+1}{2}$ edges (otherwise, the rigidity matrix associated to any embedding of $G$ into $\RR^d$ has rank smaller than $dn-\binom{d+1}{2}$, so $G$ is not $d$-rigid). Since the expected number of edges in $G(n,p)$ is $p\binom{n}{2}$, if $G(n,p)$ is $d$-rigid \whp{} then it must satisfy the inequality $p \binom{n}{2} \geq dn-d(d+1)/2$, which holds for $d\leq (1-o(1))(1-\sqrt{1-p})n$. Accordingly, we conjecture:

\begin{conjecture}\label{conj:gnp}
  For every $p=\omega(\log{n}/n)$,
  $G\sim G(n,p)$ is \whp{} $d$-rigid for $d=(1-o(1))(1-\sqrt{1-p})n$.
\end{conjecture}

Next, we establish the following equivalent of \cref{thm:gnp} for random bipartite graphs.
Let $G(n,n,p)$ be the \defn{binomial bipartite random graph}, which is the distribution over bipartite graphs with two parts $A$ and $B$, each of size $n$, 
in which each vertex in $A$ is adjacent to each vertex in $B$ independently with probability $p=p(n)$.

\begin{theorem}[Binomial random bipartite graphs]\label{thm:gnnp}
  For every $\eps>0$ there exists $c>0$ such that if $p\ge (1+\eps)\log{n}/n$, then $G\sim G(n,n,p)$ \whp{} admits a strong bipartite $d$-rigid partition for
  $d=\min\{cnp/\log(np),0.8\sqrt{n}\}$.
  
\end{theorem}
The upper bound $O(\sqrt{n})$\footnote{%
Note that we have not made any particular effort to improve the value $0.8$ in the theorem.}
on $d$ in \cref{thm:gnnp} should not come as a surprise,
given that the complete bipartite graph $K_{n,n}$ is $d$-rigid
for $d\sim 2\sqrt{n}$; see \cref{prop:complete:bipartite} and the preceding discussion.
Nonetheless, it is a noteworthy observation that relatively sparse binomial random bipartite graphs (for $p=\Omega(\log{n}/\sqrt{n})$)
attain this maximal possible rigidity, at least up to constant factors.
This stands in contrast to what is seen in non-bipartite graphs,
as indicated by \cref{thm:gnp,conj:gnp}.

Since $1$-rigidity is equivalent to connectivity,
the notion of $d$-rigidity naturally arises as a generalisation for connectivity.
A fundamental phenomenon in random graphs
is the ``phase transition'',
where a ``giant'' (linear-sized) connected component emerges.
Extending this concept,
it is natural to investigate the (sudden?) emergence of a giant $d$-rigid component.
A \defn{$d$-rigid component} of a graph $G=(V,E)$
is a maximal $d$-rigid induced subgraph of $G$.
That is, for $U\subseteq V$,
$G[U]$ is a $d$-rigid component of $G$ if it is $d$-rigid
but $G[U']$ is not $d$-rigid for any $U\subsetneq U'$.
It is conjectured by Lew, Nevo, Peled and Raz~\cite{LNPR23}
(and proved for $d=2$ by Kasiviswanathan, Moore and Theran in~\cite{KMT11}),
that the random graph $G(n,p)$
undergoes a phase transition for the containment of a giant $d$-rigid component
exactly at the ``time'' that the average degree of the $(d+1)$-core of the graph exceeds $2d$.
This coincides with the threshold for $d$-orientability of $G(n,p)$ (see \cite{FR07}),
which happens at $p\sim d/n$ (see \cite{DM09}*{Theorem~1}).

In this work, we establish that for sufficiently large $c=c(d)$,
the random graph $G(n,c/n)$ contains, \whp{},
a giant --- and in fact almost spanning --- $d$-rigid component.

\begin{theorem}[Giant rigid component]\label{thm:rigid_component}
  For every $\eps\in(0,1)$ there exists $C>0$
  such that for every fixed $d\ge 2$,
  $G\sim G(n,Cd\log{d}/n)$ contains \whp{} a $d$-rigid component with at least $(1-\eps)n$ vertices.
\end{theorem}

\subsection{Applications to dense graphs}
Next, we consider the rigidity of dense graphs.
More specifically, we look at ``Dirac graphs'', that is, $n$-vertex graphs with minimum degree at least $n/2$.

\begin{theorem}
\label{thm:dirac}
  Let $\ell=\ell(n)$ satisfy $3\log{n}/2\leq \ell< n/2$.
  Let $G=(V,E)$ be an $n$-vertex graph with $\delta(G)\ge \frac{n}{2}+\ell$.
  Then,
  $G$ admits a strong $d$-rigid partition
  for $d= 2\ell/(3\log{n})$.
\end{theorem}

\Cref{thm:dirac} is an immediate consequence of a more general result about
graphs in which every pair of vertices has ``many'' common neighbours;
see~\cref{thm:common_neighbours}.
For $\ell=\Theta(n)$, \cref{thm:dirac} is sharp up to constants.
Consider, for instance,
$G\sim G(n,p)$ for $p=\frac{1}{2}+2\eps$;
here,
$\delta(G)\ge \left(\frac{1}{2}+\eps\right)n$ \whp{}.
However,
the largest $d$ for which $G$ admits a strong $d$-rigid partition
is \whp{} $O(n/\log{n})$.
For smaller values of $\ell$, the following theorem offers a significantly improved bound
on the rigidity of the graph.
\begin{theorem}\label{thm:dirac:small}
  There exists $c>0$ such that
  for every $0\le d< c\sqrt{n}/\log{n}$,
  if $G$ is an $n$-vertex graph with $\delta(G)\ge \frac{n}{2}+d$
  then $G$ is $(d+1)$-rigid.
\end{theorem}
\cref{thm:dirac:small} is sharp up to a multiplicative constant.
For example,
consider a graph composed of two cliques,
each of size $n/2+k$, intersecting in $2k$ vertices.
This graph, an $n$-vertex graph with a minimum degree of $n/2+k-1$,
has vertex-connectivity of $2k$ and, hence, is at most $2k$-rigid
(and in fact, it is easy to check that its rigidity is exactly $2k$).
It is also sharp in terms of the minimum degree,
as there are disconnected $n$-vertex graphs with $\delta(G)=\lfloor n/2\rfloor-1$.
We conjecture that for all $1\le \ell\le n/2-1$,
every graph with a minimum degree of at least $n/2+\ell$ is $\Omega(\ell)$-rigid:
\begin{conjecture}\label{conj:dirac}
  Let $1\leq \ell(n)<n/2$.
  Let $G$ be an $n$-vertex graph with minimum degree $\delta(G)\ge\frac{n}{2}+\ell(n)$.
  Then, $G$ is $d$-rigid for $d=\Omega(\ell(n))$.
\end{conjecture}

\vspace{1em}

\paragraph{Notation and terminology}
Let $G=(V,E)$ be a graph.
For two (not necessarily disjoint) vertex sets $A$, $B$,
we let $E(A, B)$
be the set of edges having one endpoint in $A$ and the other in $B$,
and write $E(A)=E(A,A)$.
We denote further $e(A,B)=\left|\left\{(u,v)\in A\times B:\ \{u,v\}\in E\right\}\right|$.
Note that if $A\cap B\ne\es$ then it might be the case
that $e(A,B)>|E(A,B)|$.
We denote by $N(A)$ the
{\em external} neighbourhood of $A$,
that is, the set of all vertices in $V\sm A$
that have a neighbour in $A$.
In the above notation we often replace $\{v\}$ with $v$ for abbreviation.
The degree of a vertex $v\in V$,
denoted by $\deg(v)$,
is its number of incident edges.
We let $\delta(G)$ and
$\Delta(G)$ be the minimum and maximum degrees
of $G$, respectively.

Throughout the paper, all logarithms are in the natural basis.
If $f,g$ are functions of $n$,
we write $f\ll g$ if $f=o(g)$,
$f\gg g$ if $f=\omega(g)$,
and $f\sim g$ if $f=(1+o(1))g$.
For the sake of clarity of presentation,
we will systematically omit floor and ceiling signs.

\paragraph{Paper organization}
The structure of this paper is as follows.
In \cref{sec:rigid:partitions}
we present and prove the quantitative counterpart of \cref{thm:rigid:partition}
(\cref{thm:rigid:partition:quant}),
and subsequently detail the proofs of 
\cref{cor:cdsp:rigid,cor:strong:rigid:partition,cor:strong:bipartite:rigid:partition}
in \cref{sec:rigid:partitions:special},
of \cref{thm:connectivity}
in \cref{sec:conn},
and of \cref{prop:complete:bipartite}
in \cref{sec:bipartite}.
In \cref{sec:remarks} we discuss some of the limitations of $d$-rigid partitions, and present examples of $d$-rigid graphs that do not admit $d$-rigid partitions (for certain values of $d$).
\Cref{sec:conditions}, which is self-contained,
introduces general easily verifiable sufficient conditions
for the existence of rigid partitions.
In \cref{sec:randomgraphs},
we delve into the applications of this method to various (pseudo)random graphs, 
presenting the proofs of
\cref{thm:nkl,thm:gnr,thm:gnt,thm:gnp,thm:gnnp,thm:rigid_component}.
We conclude in \cref{sec:dirac} with the proof of \cref{thm:dirac,thm:dirac:small}.

\section{Rigidity via rigid partitions}\label{sec:rigid:partitions}

\subsection{Stiffness matrices and $d$-dimensional algebraic connectivity}\label{sec:stiffness}

Let $G=(V,E)$ be a graph and $\p:V\to \RR^d$ be an embedding. For $u,v\in V$, let 
\[
d_{uv}(\p)=\frac{\p(u)-\p(v)}{\|\p(u)-\p(v)\|}\in \RR^d.
\]
We define the \defn{normalized rigidity matrix} $\Rhat(G,\p)\in \RR^{d|V|\times |E|}$ as follows: the columns of $\Rhat(G,\p)$ are indexed by the edges of $G$, and the rows are indexed by assigning $d$ consecutive coordinates to each vertex of $V$. For an edge $\{u,v\}\in E$, the column of $\Rhat(G,p)$ indexed by $\{u,v\}$  is the vector $\bv_{u,v}\in \RR^{d|V|}$ be defined by
\[
   \bv_{u,v}\trans= \kbordermatrix{
     & &  &  & u & & & & v & & & \\
     & 0 & \ldots & 0 & d_{uv}(\p) \trans & 0 & \ldots & 0 & d_{vu}(\p)\trans & 0 & \ldots & 0}.
\]
That is, the vector $\bv_{u,v}$ is supported on the coordinates associated to the vertices $u$ and $v$, where it is equal to the $d$-dimensional vectors $d_{uv}$ and $d_{vu}$ respectively.

Note that $\Rhat(G,\p)$ can be obtained from the unnormalized rigidity matrix $\R(G,\p)$ by first dividing each row of $\R(G,\p)$ by a non-zero constant $\|\p(u)-\p(v)\|$, and then by taking the transpose of the matrix. In particular, $\Rhat(G,\p)$ and $\R(G,\p)$ have the same rank.

The \defn{stiffness matrix} of the framework $(G,\p)$ is the matrix
\[
    \sL(G,\p)= \Rhat(G,\p) \Rhat(G,\p)\trans \in \RR^{d|V|\times d|V|}.
\]
For a symmetric matrix $M\in \RR^{m\times m}$ and $1\leq i\leq m$, let $\lambda_i(M)$ be the $i$-th smallest eigenvalue of $M$. Note that $\sL(G,\p)$ is a positive semi-definite matrix satisfying
$\rank(\sL(G,\p))=\rank(\Rhat(G,\p))=\rank(\R(G,\p))\leq d|V|-\binom{d+1}{2}$. Therefore, we have $\lambda_i(\sL(G,\p))=0$ for $1\leq i\leq \binom{d+1}{2}$, and $(G,\p)$ is infinitesimally rigid if and only if $\lambda_{\binom{d+1}{2}+1}(\sL(G,\p))>0$. 

In~\cite{JT22}, Jord\'an and Tanigawa introduced the \defn{$d$-dimensional algebraic connectivity of $G$}, defined by
\[
    \na_d(G)= \sup \left\{ \lambda_{\binom{d+1}{2}+1}(\sL(G,\p))\right\},
\]
where the supremum is taken over all embeddings $\p:V\to \RR^d$.\footnote{The original definition in~\cite{JT22} allows $\p$ to be non-injective, however both definitions are equivalent
(see~\cite{LNPR22+}*{Lemma~2.4}).} 

Notice that $\na_d(G)>0$ if and only if $G$ is $d$-rigid. 
We can think of $\na_d(G)$ as a ``quantitative measure'' of the $d$-rigidity of $G$
--- we expect graphs with large $\na_d(G)$ to be more ``strongly rigid",
in certain senses (see e.g.~\cite{JT22}*{Corollary~8.2} for a result quantifying such a phenomenon).

For $d=1$, it is easy to check that, for any embedding $\p:V\to \RR$, the normalized rigidity matrix $\Rhat(G,\p)$ is just the incidence matrix of the orientation of $G$ induced by the embedding, and the stiffness matrix $\sL(G,\p)$ is exactly the graph Laplacian $\sL(G)$.
Therefore, the one-dimensional algebraic connectivity of $G$ coincides with the classical notion of algebraic connectivity (a.k.a.\ Laplacian spectral gap) of a graph,
introduced by Fiedler in~\cite{Fie73}.
We denote in this case $\na_1(G)=\na(G)=\lambda_2(\sL(G))$. For convenience, we define $\na(G)=\infty$ for the graph $G=(\{v\},\es)$.

We can now state the following quantitative version of \cref{thm:rigid:partition}.

\begin{theorem}\label{thm:rigid:partition:quant}
  Let $G=(V,E)$ be a graph, and let $\left(\{V_i\}_{i=1}^{d+1}, \{E_{ij}\}_{1\leq i<j\leq d+1}\right)$
  be a $d$-rigid partition of $G$. For $1\leq i<j\leq d+1$, let $G_{ij}=(V_i\cup V_j, E_{ij})$.
  Then,
  \[
      \na_d(G)\geq \min \left\{ \frac{ \na(G_{ij})}{2} :\, 1\leq i<j\leq d+1\right\}.
  \]
  In particular, since $G_{ij}$ is connected for all $1\leq i<j\leq d+1$, $G$ is $d$-rigid.
\end{theorem}

\subsection{Limit frameworks}\label{sec:limit:frameworks}

For the proof of \cref{thm:rigid:partition:quant},
we will need to consider the limit of a sequence of frameworks.
In order to do that, we introduce the notion of \emph{generalized frameworks}, first studied by Tay in~\cite{Tay93}
(although we use here slightly different definitions and notation; see also~\cite{SS02}). 

Let $\SS^{d-1}\subset\RR^d$ denote the $(d-1)$-dimensional unit sphere.
Let $G=(V,E)$ be a graph,
and let $\q: \{(u,e):\, e\in E, u\in e\}\to \SS^{d-1}$.
The pair $(G,\q)$ is called a \defn{generalized $d$-dimensional framework}.
For $e=\{u,v\}\in E$, we define $\bv_{u,v}\in \RR^{d|V|}$ by
\[
   \bv_{u,v}\trans= \kbordermatrix{
     & &  &  & u & & & & v & & & \\
     & 0 & \ldots & 0 & \q(u,e)\trans & 0 & \ldots & 0 & \q(v,e)\trans & 0 & \ldots & 0}.
\]
We define the \defn{rigidity matrix} $\Rhat(G,\q)\in \RR^{d|V|\times |E|}$
as the matrix whose columns are the vectors $\bv_{u,v}$ for all $\{u,v\}\in E$.
We define the \defn{stiffness matrix} of $(G,\q)$ as
\[
    \sL(G,\q)=\Rhat(G,\q) \Rhat(G,\q)\trans \in \RR^{d|V|\times d|V|}.
\] 
Similarly, we define the \defn{lower stiffness matrix} of $(G,\q)$ as
\begin{equation}\label{eq:Lminus}
    \sL^{-}(G,\q)=\Rhat(G,\q)\trans \Rhat(G,\q) \in \RR^{|E|\times |E|}.
\end{equation}
We have the following explicit description of the lower stiffness matrix, which is an immediate consequence of \eqref{eq:Lminus}.
\begin{lemma}\label{lemma:generalized_lower_stiffness}
Let $(G,\q)$ be a generalized $d$-dimensional framework. Let $e_1,e_2\in E$. Then,
\[
    \sL^{-}(G,\q)(e_1,e_2)= \begin{cases}
                    2 & \text{ if } e_1=e_2,\\
    \q(u,e_1)\cdot \q(u,e_2) & \text{ if } e_1=\{u,v\} \text{ and  } e_2=\{u,w\},\\
                    0 & \text{ otherwise.}
                    \end{cases} 
\]
\end{lemma}
Note that the non-zero eigenvalues of $\sL(G,\q)$ coincide with those of $\sL^{-}(G,\q)$. In particular, assuming $|E|\geq d|V|-\binom{d+1}{2}$, we have
\begin{equation}\label{eq:eigen}
    \lambda_{k}(\sL(G,\q))= \lambda_{|E|-d|V|+k}(\sL^{-}(G,\q)),
\end{equation}
for all $k\geq \binom{d+1}{2}+1$.

Given $G=(V,E)$ and an embedding $\p:V\to \RR^d$, we can define a generalized framework $(G,\q)$ by
\begin{equation}\label{eq:generalized_from_regular_framework}
    \q(u,\{u,v\})=d_{uv}(\p)= \frac{\p(u)-\p(v)}{\|\p(u)-\p(v)\|}
\end{equation}
for all $\{u,v\}\in E$.
The generalized framework $(G,\q)$
is equivalent to the framework $(G,\p)$,
in the sense that $\Rhat(G,\q)=\Rhat(G,\p)$, and therefore $\sL(G,\q)=\sL(G,\p)$.
The main ``advantage" of generalized frameworks is that
they allow us to talk about limits of frameworks:

We say that a generalized framework $(G,\q)$ is a
\defn{$d$-dimensional limit framework} if there exists
a sequence of $d$-dimensional frameworks $\{(G,\p_n)\}_{n=1}^{\infty}$ such that
\[
    \q(u,\{u,v\})= \lim_{n\to\infty} d_{uv}(\p_n)
\]
for all $\{u,v\}\in E$.
We denote in this case $(G,\p_n)\to (G,\q)$.
Note that if $(G,\q)$ is a limit $d$-dimensional framework,
it must satisfy $\q(u,\{u,v\})=-\q(v,\{u,v\})$ for all $\{u,v\}\in E$.

Generalized $1$-dimensional frameworks correspond to \emph{signings} of the graph $G$
(that is, to functions
  $\eta:\{(u,e):\, e\in E, u\in e\}\to\{-1,1\}$),
and their stiffness matrix is the ``signed Laplacian" corresponding to the signing $\eta$
(for more on the theory of signed graphs and their associated matrices,
  see, e.g.,~\cite{Zas10}).
The $1$-dimensional limit frameworks are exactly those corresponding to \emph{orientations}
of $G$
  (that is, to signings satisfying $\eta(u,\{u,v\})=-\eta(v,\{u,v\})$ for all $\{u,v\}\in E$).
Actually,
any $1$-dimensional limit framework can be obtained from an embedding $\p:V\to\RR^1$
as in~\eqref{eq:generalized_from_regular_framework}
(without the need of taking a limit of a sequence of embeddings).
The stiffness matrix $\sL(G,\eta)$ in this case is just the usual graph Laplacian $\sL(G)$.

\begin{lemma}\label{lemma:limit_framework_ad}
If $(G,\q)$ is a $d$-dimensional limit framework, then
\[
\na_d(G)\geq \lambda_{\binom{d+1}{2}+1}(\sL(G,\q)).
\]
\end{lemma}
\begin{proof}
Let $(G,\p_n)\to (G,\q)$, where $\{\p_n\}_{n=1}^{\infty}$ is a sequence of embeddings of $V$ into $\mathbb{R}^d$.
Note that $\Rhat(G,\q)$ is the entry-wise limit of the sequence of matrices $\Rhat(G,\p_n)$, and therefore $\sL(G,\q)$ is the limit of the matrices $\sL(G,\p_n)$.  Hence, by the continuity of eigenvalues,
\[
    \lambda_{\binom{d+1}{2}+1}(\sL(G,\q)) =\lim_{n\to\infty} \lambda_{\binom{d+1}{2}+1}(\sL(G,\p_n)).
\]
Thus, we obtain
\[
    \na_d(G)\geq \sup_{n} \lambda_{\binom{d+1}{2}+1}(\sL(G,\p_n)) \geq \lambda_{\binom{d+1}{2}+1}(\sL(G,\q)).\qedhere
\]
\end{proof}

\begin{lemma}\label{lemma:constructing_limit_frameworks}
Let $G=(V,E)$.
Let $V=V_1\cup\cdots\cup V_k$ be a colouring of $V$,
and let $x_1,\ldots,x_k$ be $k$ distinct points in $\RR^d$.
Assume we have for each $1\leq i\leq k$ a $d$-dimensional limit framework $(G[V_i],\q_i)$.
Define $\q: \{(u,e):\, e\in E, u\in e\}\to \SS^{d-1}$ by
\[
\q(u,\{u,v\})= \begin{cases}
    (x_i-x_j)/\|x_i-x_j\| & \text{ if } u\in V_i, v\in V_j \text{ for } i\neq j,\\
    \q_i(u,\{u,v\}) & \text{ if } u,v\in V_i \text{ for } 1\leq i\leq k,    
\end{cases}
\]
for all $\{u,v\}\in E$.
Then, $(G,\q)$ is a $d$-dimensional limit framework.
\end{lemma}
\begin{proof}
For each $1\leq i\leq k$,
let $\p_{i,n} : V_i\to \RR^d$ be a sequence of embeddings such that
$(G[V_i],\p_{i,n})\to (G[V_i],\q_i)$.
We may assume without loss of generality that $\|\p_{i,n}(u)\|\leq 1$
for all $1\leq i\leq k$ and $u\in V_i$
(since re-scaling of a map $\p$ does not affect the values of $d_{uv}(\p)$).

Define $\p_n: V\to \RR^d$ by
\[
    \p_n(u)= n x_i + \p_{i,n}(u)
\]
for all $1\leq i\leq k$ and $u\in V_i$.
We will show that $(G,\p_n)\to (G,\q)$. Let $\{u,v\}\in E$. If $u,v\in V_i$ for some $1\leq i \leq k$, then
\[
    d_{uv}(\p_n)=\frac{\p_n(u)-\p_n(v)}{\|\p_n(u)-\p_n(v)\|}=  \frac{\p_{i,n}(u)-\p_{i,n}(v)}{\|\p_{i,n}(u)-\p_{i,n}(v)\|} \to \q_i(u,\{u,v\})=\q(u,\{u,v\}),
\]
as required. If $u\in V_i$ and $v\in V_j$ for some $i\neq j$, then, since $\|\p_{i,n}(u)\|\leq 1$ and $\|\p_{j,n}(v)\|\leq 1$, we have
\[
\lim_{n\to\infty }\p_{n}(u)/n= x_i
\]
and
\[
\lim_{n\to\infty }\p_{n}(v)/n = x_j.
\]
Therefore,
\[
 d_{uv}(\p_n)=\frac{\p_n(u)-\p_n(v)}{\|\p_n(u)-\p_n(v)\|}=\frac{\p_{n}(u)/n-\p_{n}(v)/n}{\|\p_{n}(u)/n-\p_{n}(v)/n\|}\to\frac{x_i-x_j}{\|x_i-x_j\|}.
  \qedhere
 \]
\end{proof}

The following lemma is a simple generalisation of an argument
by Tay~\cite{Tay93}*{Theorem~3.4} (see also~\cite{SS02}*{Theorem~8}).

\begin{lemma}\label{lemma:combing}
  Let $G=(V,E)$ be a graph, and let $E=E_1\cup \cdots\cup E_k$ be a colouring of its edge set.
  Let $y_1,\ldots, y_k\in \SS^{d-1}$.
  If every $U\subseteq V$ of size at least $2$ has a monochromatic cut,
  then there exists a $d$-dimensional limit framework $(G,\q)$ such that,
  for every $i\in [k]$,
  $\q(u,\{u,v\})\in \{y_i,-y_i\}$ for all $\{u,v\}\in E_i$.
\end{lemma}

\begin{proof}
We argue by induction on $|V|$. If $|V|=1$, we are done. Otherwise, there exist some $i\in[k]$ and a partition $V=A\cup B$ such that all the edges in $G$ between $A$ and $B$ belong to $E_i$.
By the induction hypothesis, there exist limit frameworks $(G[A],\q_A)$ and $(G[B],\q_B)$ such that for every $j\in [k]$ and $\{u,v\}\in E_j$,  $\q_A(u,\{u,v\})\in\{y_j,-y_j\}$ 
 if   $u,v\in A$, and $\q_B(u,\{u,v\})\in\{y_j,-y_j\}$ 
 if   $u,v\in B$. Let $x_A=y_i$ and $x_B=0$. Note that
 $(x_A-x_B)/\|x_A-x_B\|= y_i$.
 Therefore, by \cref{lemma:constructing_limit_frameworks},
 there is a $d$-dimensional limit framework $(G,\q)$ satisfying
 \[
 \q(u,\{u,v\})=
 \begin{cases}
     y_i & \text{ if $u\in A$ and $v\in B$,}\\
       -y_i & \text{ if $u\in B$ and $v\in A$,}\\
     \q_A(u,\{u,v\}) & \text{ if } u,v\in A,\\
     \q_B(u,\{u,v\}) & \text{ if } u,v\in B,
 \end{cases}
 \]
 for all $\{u,v\}\in E$. Note that for every $j\in[k]$ and $\{u,v\}\in E_j$, we obtain $\q(u,\{u,v\})\in \{y_j,-y_j\}$, as desired.
\end{proof}

\subsection{Proof of \cref{thm:rigid:partition:quant}}\label{sec:rigid:partitions:pf}

For the proof of \cref{thm:rigid:partition:quant}
we will need the following very simple lemma from~\cite{LNPR23+}.

\begin{lemma}[{\cite[Lemma~3.1]{LNPR23+}}]
\label{lemma:block_diagonal}
Let $M\in\RR^{n\times n}$ be a block diagonal symmetric matrix with blocks $M_1,\ldots,M_k$, where $M_i\in \RR^{n_i\times n_i}$ for every $1\leq i\leq k$. Let $1\leq m\leq n$ and $r_1,\ldots, r_k$ be such that $\sum_{i=1}^k r_i=m+k-1$. Then,
\[
\lambda_m(M)\geq \min \{ \lambda_{r_i}(M_i) : \, 1\leq i\leq k\}.
\]
\end{lemma}
For convenience, given a ``$0\times 0$" matrix $M$, we define $\lambda_1(M)=\infty$.
Note then that \cref{lemma:block_diagonal} applies even if we allow values $n_i=0$ and $r_i=1$
for one or more $1\leq i\leq k$.

\begin{proof}[Proof of \cref{thm:rigid:partition:quant}]

Without loss of generality, we assume that $\{E_{ij}\}_{1\leq i<j\leq d+1}$ is a colouring of $E$.
Otherwise, we can apply the arguments below to the subgraph $G'=(V,\cup_{1\leq i<j\leq d+1} E_{ij})$
(and then use the fact that $a_d(G)\geq a_d(G')$, see e.g.~\cite{LNPR22+}*{Theorem~2.3}).

Recall that, for $1\leq i<j\leq d+1$, we define $G_{ij}=(V_{i}\cup V_{j},E_{ij})$.
For convenience, for $i>j$, we define $G_{ij}=G_{ji}$.
 Let $x_1,\ldots,x_{d+1}\in\RR^d$ be the vertices of a regular simplex.
 For $i\neq j$, let $y_{ij}=(x_i-x_j)/\|x_i-x_j\|$. 

Let $1\leq i\leq d+1$. By \cref{lemma:combing}
(using the fact that any $U\subseteq V_i$ of size at least $2$ has a monochromatic cut),
there exists a $d$-dimensional limit framework $(G[V_i],\q_i)$,
satisfying $\q_i(u,\{u,v\})\in \{y_{ij},-y_{ij}\}$ for all $j\neq i$ and $\{u,v\}\in E_{ij}$.
Therefore, by \cref{lemma:constructing_limit_frameworks},
there exists a $d$-dimensional limit framework $(G,\q)$ satisfying,
for all $i\neq j$ and $\{u,v\}\in E_{ij}$, $\q(u,\{u,v\})\in \{y_{ij},-y_{ij}\}$
if $u,v\in V_i$ or $u,v\in V_j$, and $\q(u,\{u,v\})=y_{ij}$ if $u\in V_i$ and $v\in V_j$.

Let $\eta:\{(u,e):\, e\in E, u\in e\}\to\{1,-1\}$ be defined by
\[
    \eta(u,\{u,v\})= \begin{cases} 
    1 & \text{if } \q(u,\{u,v\})=y_{ij},\\
    -1 & \text{if } \q(u,\{u,v\})=-y_{ij}
    \end{cases}
\]
for all $\{u,v\}\in E_{ij}$ such that $u\in V_i$. Note that, if $u\in V_i$ and $v\in V_j$ for $i\neq j$, then $\eta(u,\{u,v\})=\eta(v,\{u,v\})=1$. Moreover, since $\q$ is a $d$-dimensional limit framework, we have $\q(u,\{u,v\})=-\q(v,\{u,v\})$ for all $\{u,v\}\in E$. Therefore, we have $\eta(u,\{u,v\})=-\eta(v,\{u,v\})$  whenever $u,v\in V_i$ for some $1\leq i\leq d+1$. 

It is easy to check that if $e=\{u,v\}\in E_{ij}$ and $e'=\{u,w\}\in E_{ik}$ for $i\neq j,k$, then 
\[
    \q(u,e)\cdot \q(u,e') = (y_{ij}\cdot y_{ik}) \eta(u,e)\eta(u,e').
\]
Note that, if $j\neq k$,  $y_{ij}\cdot y_{ik}=1/2$, and if $j=k$, $y_{ij}\cdot y_{ik}=1$.
Therefore, by \cref{lemma:generalized_lower_stiffness}, we obtain
\[
    \sL^{-}(G,\q)(e,e')=\begin{cases}
    2 & \text{ if } e=e',\\
    \eta(u,e)\cdot \eta(u,e') & \text{ if } e\cap e'=\{u\},\, e,e'\in E_{ij} \text{ for $i\neq j$},\\
    \frac{1}{2} \eta(u,e)\cdot \eta(u,e') & \text{ if } e\cap e'=\{u\},\, e\in E_{ij}, \, e'\in E_{ik} \text{ for pairwise distinct } i,j,k,
    \\
    0 & \text{ otherwise.}
    \end{cases}
\]
We can write $\sL^{-}(G,\q)$ as a sum
\[
\sL^{-}(G,\q)= \frac{1}{2} M+\frac{1}{2} T,
\]
where $M$ is the matrix defined by
\[
    M(e,e')=\begin{cases}
    2 & \text{ if } e=e',\\
    \eta(u,e)\cdot \eta(u,e') & \text{ if } e\cap e'=\{u\},\\
    0 & \text{ otherwise,}
    \end{cases}
\]
and $T$ is a block diagonal matrix with blocks $T_{ij}$ for $1\leq i<j\leq d+1$, where each block $T_{ij}$ is supported on the rows and columns corresponding to $E_{ij}$, and is defined by
\[
T_{ij}(e,e')=\begin{cases}
    2 & \text{ if } e=e',\\
    \eta(u,e)\cdot \eta(u,e') & \text{ if } e\cap e'=\{u\},\\
    0 & \text{ otherwise,}
    \end{cases}
\]
Note that $M=\sL^{-}(G,\eta)$, and in particular it is a positive semi-definite matrix. Thus, by Weyl's inequality
(see, e.g.,~\cite{BH}*{Theorem~2.8.1(iii)}),
\begin{equation}\label{eq:weyl1}
\lambda_i(\sL^{-}(G,\q))\geq 
 \frac{1}{2}\lambda_i(T)  \, \text{ for all $1\leq i \leq |E|$}.
\end{equation}
Let $1\leq i<j\leq d+1$. Define a signing $\eta_{ij}: \{(u,e):\, e\in E_{ij},\, u\in e\} \to \{1,-1\}$ by
\[
\eta_{ij}(u,\{u,v\})=\begin{cases}
    \eta(u,\{u,v\}) & \text{ if } u,v\in V_i,\\
    -\eta(u,\{u,v\}) & \text{ if } u,v\in V_j,\\
    1 & \text{ if } u\in V_i, v\in V_j,\\
    -1 & \text{ if } i\in V_j, v\in V_i.
\end{cases}
\]
Note that $\eta_{ij}(v,\{u,v\})=-\eta_{ij}(u,\{u,v\})$, so $\eta_{ij}$ is an orientation of $G_{ij}$, and therefore $\sL(G_{ij},\eta_{ij})$ is actually the Laplacian matrix $\sL(G_{ij})$.
Moreover, for any $\{u,v\},\{u,w\}\in E_{ij}$, 
\[
\eta(u,\{u,v\})\cdot \eta(u,\{u,w\})
=
\eta_{ij}(u,\{u,v\})\cdot \eta_{ij}(u,\{u,w\}),
\]
and therefore $T_{ij}=\sL^{-}(G_{ij},\eta_{ij})$. Let $r_{ij}=|E_{ij}|-|V_{i}\cup V_j|+2$. Since $G_{ij}$ is connected, we have $|E_{ij}|\geq |V_i\cup V_j|-1$.
So, by~\eqref{eq:eigen}, we obtain
\[
    \lambda_{r_{ij}}(T_{ij})=\lambda_{r_{ij}}(\sL^{-}(G_{ij},\eta_{ij}))=\lambda_2(\sL(G_{ij},\eta_{ij}))=\lambda_2(\sL(G_{ij}))=\na(G_{ij}).
\]
Let $m=|E|-d|V|+\binom{d+1}{2}+1$. Note that
\begin{align*}
\sum_{1\leq i<j\leq d+1} r_{ij} &= \sum_{1\leq i<j\leq d+1}|E_{ij}| -d\sum_{i=1}^{d+1}|V_i| +2\binom{d+1}{2} 
\\
&= |E|-d|V|+2\binom{d+1}{2} = m + \binom{d+1}{2}-1.
\end{align*}
So, by \cref{lemma:block_diagonal}, 
\[
\lambda_m(T) 
\geq \min\left\{ \lambda_{r_{ij}}(T_{ij}) :\, 1\leq i<j\leq d+1\right\}
= \min\left\{ \na(G_{ij}) :\, 1\leq i<j\leq d+1\right\}.
\]
By \cref{lemma:limit_framework_ad}, \eqref{eq:eigen} and~\eqref{eq:weyl1}, we obtain
\[
    \na_d(G)\geq \lambda_{\binom{d+1}{2}+1}(\sL(G,\q))= \lambda_m(\sL^{-}(G,\q))
    \geq \frac{1}{2}\lambda_m(T)\geq \min\left\{ \frac{\na(G_{ij})}{2} :\, 1\leq i<j\leq d+1\right\},
\]
as wanted.
\end{proof}

\subsection{Special rigid partitions}\label{sec:rigid:partitions:special}
We proceed to prove \cref{cor:cdsp:rigid,cor:strong:rigid:partition,cor:strong:bipartite:rigid:partition}.
\cref{cor:cdsp:rigid} follows from the next proposition together with \cref{thm:rigid:partition}.

\begin{proposition}
  A graph that admits a CDS partition of size $\binom{d+1}{2}$ admits a $d$-rigid partition.
\end{proposition}

\begin{proof}
  If $d=1$ then the claim is trivial; we therefore assume that $d\ge 2$.
  Let $\{A_{ij}:\, 1\leq i<j\leq d+1\}$ be a partition of $V$ into $\binom{d+1}{2}$ connected dominating sets.
  Let $V_1=A_{13}$, $V_2=A_{12}$, $V_3=A_{23}$, and $V_j=\bigcup_{i<j} A_{ij}$ for $4\leq j\leq d+1$.
  For $1\leq i<j\leq d+1$, let $E_{ij}= E(V_i,V_j)\cup E(A_{ij})$, and let $G_{ij}=(V_i\cup V_j,E_{ij})$.
  For convenience, we define $E_{ji}=E_{ij}$ and $A_{ji}=A_{ij}$.

  First, we show that the graphs $G_{ij}$ are connected. 
  Let $1\leq i<j\leq d+1$.
  Assume that $A_{ij}\subseteq V_j$
  (this is the case unless $i=1$ and $j=3$, in which case the proof follows in essentially the same way).
  Since $A_{ij}$ is a dominating set, every vertex in $V_i$ is adjacent in $G_{ij}$ to a vertex in $A_{ij}$. Similarly, since $V_i$ is a dominating set (as it contains a dominating set), every vertex in $V_j$ is adjacent in $G_{ij}$ to a vertex in $V_i$. Therefore, every vertex in $V_i\cup V_j$ is connected in $G_{ij}$ to a vertex in $A_{ij}$ (by a path of length at most two). Since $A_{ij}$ is connected in $G_{ij}$, this implies that $G_{ij}$ is connected.

  We proceed to show that for every $1\leq i\leq d+1$,
  every $U\subseteq V_i$ of size at least $2$ has a monochromatic cut
  (that is, a partition into $U',U''$ such that $\hat{E}\cap E(U',U'')\subseteq E_{ij}$ for some $j\neq i$,
  where $\hat{E}=\bigcup_{1\leq i'<j'\leq d+1} E_{i'j'}$).
  But this is trivially true: Let $u\in U$. Then $u\in A_{ij}$ for some $j\neq i$. Let $U'=\{u\}$ and $U''=U\smallsetminus \{u\}$. Then, $\hat{E}\cap E(U',U'')\subseteq E(A_{ij})\subseteq E_{ij}$.
  Therefore,  $\left(\{V_i\}_{i=1}^{d+1}, \{E_{ij}\}_{1\leq i<j\leq d+1}\right)$ is a $d$-rigid partition of $G$.
\end{proof}

\Cref{cor:strong:rigid:partition} follows from the next two propositions
together with \cref{thm:rigid:partition}.
Recall that, given a colouring $V_1,\dots,V_{d+1}$ of $V$
and a partial colouring $\{E_{ij}\}_{1\le i<j\le d+1}$ of $E$, 
we denote by $G_{ij}$ the subgraph $(V_i\cup V_j,E_{ij})$ of $G$
and set $\hat{E}=\bigcup_{1\le i<j\le d+1} E_{ij}$.

\begin{proposition}
  A graph that admits a type I strong $d$-rigid partition admits a $d$-rigid partition.
\end{proposition}

\begin{proof}
  Let $V_1,\dots,V_d$ be a type I strong $d$-rigid partition of a graph $G=(V,E)$.
  Let $V_{d+1}=\es$, so $V_1,\dots,V_{d+1}$ is a colouring of $V$.
  For $1\le i<j\le d$ let $E_{ij}=E(V_i,V_j)$,
  and for $1\le i\le d$ and $j=d+1$ let $E_{ij}=E(V_i)$.
  Thus, $\{E_{ij}\}_{1\le i<j\le d+1}$ is a partial colouring of $E$.
  Moreover,
  for $1\le i<j\le d$, $G_{ij}$ is connected since $G_{ij}=G[V_i,V_j]$,
  and for $1\le i\le d$ and $j=d+1$, $G_{ij}$ is connected since $G_{ij}=G[V_i]$.
  Finally, let $1\le i\le d+1$ and let $U\subseteq V_i$ with $|U|\ge 2$
  (so $i\le d$).
  We have to show that $U$ has a monochromatic cut.
  However, that follows directly from the fact that $E(U)\subseteq E_{i,d+1}$.
\end{proof}

\begin{proposition}
  A graph that admits a type II strong $d$-rigid partition admits a $d$-rigid partition.
\end{proposition}

\begin{proof}
  Let $V_1,\dots,V_{d+1}$ be a type II strong $d$-rigid partition of a graph $G=(V,E)$.
  So $V_1,\dots,V_{d+1}$ is a colouring of $V$.
  For $1\le i<j\le d+1$ let $E_{ij}=E(V_i,V_j)$.
  Thus, $\{E_{ij}\}_{1\le i<j\le d+1}$ is a partial colouring of $E$.
  Moreover,
  for $1\le i<j\le d+1$, $G_{ij}$ is connected since $G_{ij}=G[V_i,V_j]$.
  Finally, let $1\le i\le d+1$ and let $U\subseteq V_i$ with $|U|\ge 2$.
  We have to show that $U$ has a monochromatic cut.
  However, that follows directly from the fact that
  $\hat{E}\cap E(V_i)=\es$ for every $1\le i\le d+1$.
\end{proof}

\Cref{cor:strong:bipartite:rigid:partition} follows
from the next proposition together with \cref{thm:rigid:partition}.

\begin{proposition}
  A graph that admits a strong bipartite $d$-rigid partition admits a $d$-rigid partition.
\end{proposition}

\begin{proof}
  Let $V_1,\dots,V_{d+1}$ be a strong bipartite $d$-rigid partition
  of a bipartite graph $G=(A\cup B,E)$ with bipartition $(A,B)$.
  So $V_1,\dots,V_{d+1}$ is a colouring of $V=A\cup B$.
  For $1\le i\le d+1$, write $A_i=V_i\cap A$ and $B_i=V_i\cap B$.
There is a sequence $0\le s_1\leq s_2\le \cdots \le s_{d+1}$ satisfying $\sum_{i=1}^k s_i\geq \binom{k}{2}$ for all $1\leq k\leq d+1$ such that $G[A_i,B_i]$ contains a forest $F_i$ with $s_i$ edges for all $1\leq i\leq d+1$. Without loss of generality, we may assume that $\sum_{i=1}^{d+1}s_i=\binom{d+1}{2}$. 
By a theorem of Landau~\cite{Lan53}, there exists a tournament $T$ on vertex set $[d+1]$ such that the out-degree of vertex $i$ is exactly $s_i$ for all $1\leq i\leq d+1$.  
For each $1\leq i\leq d+1$,
denote the edges of $F_i$ by $e_{i}^{j_1},\ldots,e_{i}^{j_{s_i}}$,
where $j_1,\ldots,j_{s_i}$ are the $s_i$ out-neighbours of $i$ in $T$. 

Let $1\le i<j\le d+1$.
If $(i,j)$ is a directed edge of $T$, set $e_{ij}=e_i^j$;
otherwise set $e_{ij}=e_j^i$.
In any case, define
$E_{ij}=E(V_i,V_j)\cup \{e_{ij}\}$. 
Thus, $\{E_{ij}\}_{1\le i<j\le d+1}$ is a partial colouring of $E$.
  Moreover,
  for $1\le i<j\le d+1$, $G_{ij}=(V_i\cup V_j,E_{ij})$ is connected since $G[V_i,V_j]$
  is a subgraph of $G_{ij}$ with two connected components
  (as $G[A_i,B_j]$ and $G[A_j,B_i]$ are both connected)
  and $e_{ij}$ is an edge connecting them.

Finally, let $1\le i\le d+1$ and let $U\subseteq V_i$ with $|U|\ge 2$.
  We have to show that $U$ has a monochromatic cut.
 Let $\hat{E}=\cup_{1\leq i'<j'\leq d+1} E_{i'j'}$.
  If $U$ does not span any of the edges $e_i^j$ for $j\neq i$
  then $\hat{E}\cap E(U)=\es$, hence $U$ has a monochromatic cut trivially.
  Otherwise, suppose $e_i^j\subseteq U$ for some $j\neq i$,
  and write $e_i^j=\{u,v\}$. Let $U'$ be the connected component of $u$ in $F_i\smallsetminus e_i^j$, and $U''=U\smallsetminus U'$. Since $F_i$ is a forest, we have $v\in U''$, and therefore $U',U''\neq \es$. 
  Then, $\hat{E}\cap E(U',U'')=\{e_i^j\}$, which is a subset of either $E_{ij}$ (if $i<j$) or $E_{ji}$ (if $j<i$).
  Thus, $U',U''$ is a monochromatic cut for $U$, as wanted.
\end{proof}

For our arguments on random bipartite graphs, we will use the following special class of strong bipartite $d$-rigid partitions:
\begin{lemma}\label{lem:matching:bipartite}
  Let $G=(A\cup B,E)$ be a bipartite graph with bipartition $(A,B)$.
  Let $V_1,\ldots, V_{d+1}$ be a partition of $V=A\cup B$.
  Write $A_i=V_i\cap A$ and $B_i=V_i\cap B$ for $1\le i\le d+1$.
  Assume that $G[A_i,B_j]$ is connected for all $i\neq j$,
  and that $G[A_i,B_i]$ contains a matching of size $d$ for all $1\leq i\leq d+1$.
  Then, $V_1,\ldots,V_{d+1}$ is a strong bipartite $d$-rigid partition of $G$.
\end{lemma}
\begin{proof}
    The claim follows immediately from the definition of a strong bipartite $d$-rigid partition, using the fact that a matching is a forest, and that for all $1\leq k\leq d+1$ we have $k d \geq \binom{k}{2}$.
\end{proof}

\subsection{Rigidity of highly-connected graphs}\label{sec:conn}
For the proof of \cref{thm:connectivity},
we will need the following result by Censor-Hillel et al.~\cite{CGGHK17},
showing that highly-connected graphs admit a large CDS partition.
\begin{theorem}[{\cite{CGGHK17}*{Corollary~1.6}}]\label{thm:connectivity:implies:cdp}
  There exists $c>0$ such that every $k$-connected $n$-vertex graph
  admits a CDS partition of size $ck/{\log^2}{n}$.
\end{theorem}

\begin{proof}[Proof of \cref{thm:connectivity}]
  Let $c>0$ be the constant from \cref{thm:connectivity:implies:cdp}, and let $C=1/c$. 
  Let $d=d(n)\geq 1$, and let $k= C d^2{\log^2}{n}$.
  Let $G=(V,E)$ be a $k$-connected $n$-vertex graph.
  By \cref{thm:connectivity:implies:cdp},
  $G$ admits a CDS partition of size $c k/{\log^2}{n} = d^2\geq \binom{d+1}{2}$.
  Hence, by \cref{cor:cdsp:rigid}, $G$ is $d$-rigid.
\end{proof}

Note that improving the bound in \cref{thm:connectivity:implies:cdp}
would immediately imply an improved bound in \cref{thm:connectivity}.
However, Censor-Hillel, Ghaffari and Kuhn showed in \cite{CGK14}
that there exist $k$-connected graphs that can be partitioned
into at most $O(k/\log{n})$ connected dominating sets.
Therefore, the best that we can hope for, using our current arguments,
is to reduce the requirement on the connectivity of $G$ in \cref{thm:connectivity}
from $\Omega(d^2 {\log^2}{n})$ to $\Omega(d^2\log{n})$.
It is also worth noting that for every {\em fixed} $k$,
there exists a $k$-connected graph
with no CDS partition of size $2$ ---
this follows from, e.g.,~\cite{KO03}*{Proposition~6}.

\subsection{Rigidity of complete bipartite graphs}\label{sec:bipartite}
We conclude this section with a proof of Proposition \ref{prop:complete:bipartite}.

\begin{proof}[Proof of Proposition \ref{prop:complete:bipartite}]
  We define for $1\leq i\leq d+1$,
  \[
      s_i=\begin{cases}
          1 & \text{if } i\leq 3,\\
          i-1 & \text{if } i>3.
      \end{cases}
  \]
  Note that $0\leq s_i\leq s_{i+1}$ for all $1\leq i\leq d$,
  $\sum_{i=1}^k s_i \geq \binom{k}{2}$ for all $1\leq k \leq d+1$,
  and $\sum_{i=1}^{d+1}s_i=\binom{d+1}{2}$. 
  Let $V=A\cup B$ be the vertex set of $G=K_{m,n}$, and let $A$ and $B$ be its two parts,
  with $|A|=m$ and $|B|=n$.
  Since $m,n\geq d+1$ and $|V|=m+n\geq \binom{d+2}{2}= \sum_{i=1}^{d+1} (s_i+1)$,
  there is a partition $V_1,\ldots, V_{d+1}$ of $V$
  such that $V_i\cap A\neq \es$, $V_i\cap B\neq \es$ and $|V_i|\geq s_i+1$ for all $1\leq i\leq d+1$.

  For $1\leq i\leq d+1$, let $A_i=V_i\cap A$ and $B_i=V_i\cap B$.
  Note that $G[V_i]=G[A_i,B_i]$ is connected for all $1\leq i\leq d+1$
  (since it is a complete bipartite graph itself, with parts $A_i,B_i\neq \es$).
  Therefore, it contains a spanning tree with $|V_i|-1\geq s_i$ edges.
  Moreover, for any $i\neq j$, $G[A_i,B_j]$ is connected
  (again, since it is a complete bipartite graph with parts $A_i,B_j\neq \es$).
  Therefore, $V_1,\ldots,V_{d+1}$ is a strong bipartite $d$-rigid partition of $G=K_{m,n}$.
\end{proof}

\subsection{Limitations of rigid partitions}\label{sec:remarks}
We do not expect the existence of a $d$-rigid partition to be a necessary condition for $d$-rigidity, except for $d=2$.
Indeed, we have the following restriction on the maximum $d$ for which a graph $G$ admits a $d$-rigid partition:

\begin{proposition}\label{prop:restriction}
Let $G=(V,E)$ be a graph with clique number $\omega(G)$. Assume that $G$ admits a $d$-rigid partition. Then
\[
    d\leq \frac{|V|+\omega(G)}{2}.
\]
\end{proposition}
\begin{proof}
    Let $\left(\{V_{i}\}_{i=1}^{d+1},\{E_{ij}\}_{1\leq i<j\leq d+1}\right)$ be a $d$-rigid partition of $G$.
    For $1\leq i< j\leq d+1$, let $G_{ij}=(V_i\cup V_j,E_{ij})$.  
    First, note that at most one of the sets $V_i$ may be empty.
    Let $I=\{1\leq i\leq d+1:\, |V_i|= 1\}$.
    Note that, for $i,j\in I$, with $i<j$, the unique vertex in $V_i$ must be adjacent in $G$ to the unique vertex in $V_j$,
    otherwise the graph $G_{ij}$ would be disconnected.
    Therefore, the vertices in $\bigcup_{i\in I}V_i$ form a clique in $G$.
In particular, we must have $|I|\leq \omega(G)$. On the other hand, we have
    \[
        0\cdot 1+ 1\cdot |I|+ 2\cdot (d+1-|I|-1)\leq \sum_{i=1}^{d+1}V_i = |V|,
    \]
    and therefore $|I|\geq 2d-|V|$. We obtain $d\leq (|V|+\omega(G))/2$, as wanted.
\end{proof}

Now, let $n$ be an even number, and let $G_n$ be the \defn{$n$-hyperoctahedral graph}
--- that is, the graph obtained by removing a perfect matching from the complete graph on $n$ vertices.
It is conjectured that the rigidity of $G_n$ is $n-1-\lfloor \sqrt{n}+1/2\rfloor$
(see~\cite{CN23}*{Conjecture~3.4}).
We verified this conjecture by computer calculation for $n\leq 150$.
On the other hand, since a maximal clique in $G_n$ is of size $n/2$,
then by \cref{prop:restriction} $G_n$ does not admit a $d$-rigid partition for $d>\frac{3n}{4}$.

Let us consider another example that, assuming the validity of \cref{conj:gnp},
highlights an even greater disparity between graph rigidity and the existence of rigid partitions.
Consider the binomial random graph $G\sim G(n,1/2)$
(any constant edge probability would serve similarly).
If $G$ has a $d$-rigid partition $\left(\{V_{i}\}_{i=1}^{d+1},\{E_{ij}\}_{1\leq i<j\leq d+1}\right)$,
then there must be an edge in $G$ between each pair $(V_i,V_j)$
for $G_{ij}=(V_i\cup V_j,E_{ij})$ to be connected.
However, standard calculations show (see~\cite{BCE80}) that for $d\ge Cn/\sqrt{\log{n}}$
(for some $C>0$),
such a family of $d$ disjoint vertex sets does not exist \whp{}.
Therefore, if we accept the validity of \cref{conj:gnp},
it follows that $G$ is \whp{} $d$-rigid for $d=\Theta(n)$,
yet it does not admit a $d$-rigid partition for any $d\ge Cn/\sqrt{\log{n}}$.

It would be interesting to find examples of $d$-rigid graphs that do not admit $d$-rigid partitions for small values of $d$, and in particular for $d=3$.

\section{Conditions for rigid partitions}\label{sec:conditions}
In this section,
we develop general and easily verifiable sufficient conditions for the existence of rigid partitions.
Central to our findings is the notion that ``sparse'' graphs
that exhibit certain weak expansion properties admit strong rigid partitions.
In \cref{sec:randomgraphs}, this observation will be applied to various scenarios.

We will repeatedly make use of the following version of Chernoff bounds
(see, e.g., in~\cite{JLR}*{Chapter~2}).
Let $\phi(x)=(1+x)\log(1+x)-x$ for $x>-1$.
\begin{theorem}[Chernoff bounds]\label{thm:chernoff}
  Let $n\ge 1$ be an integer and let $p\in[0,1]$,
  let $X\sim\Bin(n,p)$, and let $\mu=\E{X}=np$.
  Then, for every $\nu>0$,
  \begin{align*}
    \pr(X\le \mu-\nu) &\le \exp\left(-\mu\phi\left(-\frac{\nu}{\mu}\right)\right) \le \exp\left(-\frac{\nu^2}{2\mu}\right),\\
    \pr(X\ge \mu+\nu) &\le \exp\left(-\mu\phi\left(\frac{\nu}{\mu}\right)\right) \le \exp\left(-\frac{\nu^2}{2(\mu+\nu/3)}\right),
  \end{align*}
  from which it follows, in particular, that for $\nu<\mu$,
  \[
    \pr(|X-\mu|\ge\nu) \le 2\exp\left(-\frac{\nu^2}{3\mu}\right).
  \]
\end{theorem}

We will also need the following version of the Lov\'asz Local Lemma;
see, e.g.,~\cite{AS}*{Section~5.1}.
\begin{theorem}[Lov\'asz Local Lemma]\label{thm:lll}
  Let $\cA_1,\dots,\cA_n$ be events in some probability space.
  Suppose that for every $i\in[n]$,
  each $\cA_i$ is mutually independent of a set of all other events $A_j$ but at most $\Delta$,
  and that $\pr(\cA_i)\le 1/(e(\Delta+1))$.
  Then:
  \[
    \pr\left(
      \bigcap_{i=1}^n \overline{\cA_i}
    \right) \ge \left(1-1/(\Delta+1)\right)^n.
  \]
\end{theorem}
Our next lemma is the key lemma of this section
and plays a crucial role throughout the rest of the paper.
In essence, it asserts that for any graph with a sufficiently large minimum degree
and a bounded min-to-max degree ratio,
there exists a partition of its vertex set into $d$ parts in which every vertex has an ample number of neighbours
in each part.

\begin{lemma}[Partition lemma]\label{lem:partition}
  For every $\nlb\in(0,1)$ and $\drat\ge 1$ there exists $C=C(\nlb,\drat)>1$ 
  such that the following holds.
  For every $d\ge 2$,
  any $n$-vertex graph $G=(V,E)$
  with minimum degree $\delta\ge C d\log{d}$
  and maximum degree $\Delta\le\drat\delta$
  has a partition $V=V_1\cup\cdots\cup V_d$
  such that every vertex in $V$ has at least $(1-\nlb)\delta/d$ neighbours in $V_i$ for all $1\leq i\leq d$.
  
  Moreover, the probability of obtaining such a partition, if we partition the vertex set at random by independently assigning each vertex $v\in V$ to the set $V_i$ with probability $1/d$, is at least $(1-1/\Delta^2)^n$.
\end{lemma}

\begin{proof}
Let $\{m(v) : \, v\in V\}$ be a family of independent random variables, each uniformly distributed on $[d]$. For all $1\leq i\leq d$, let $
    V_i= \{v\in V: m(v)=i\}$.
For $v\in V$ and $i\in[d]$ let $\cA_{v,i}$ be the event that $|N(v)\cap V_i|\le (1-\nlb)\delta/d$,
and let $\cA_v=\bigcup_{i\in[d]} \cA_{v,i}$.
Note that $|N(v)\cap V_i|\sim \Bin(\deg(v),1/d)$,
so by Chernoff's bounds (\cref{thm:chernoff}) we obtain
\[
    \pr(\cA_{v,i})= \pr(|N(v)\cap V_i|\le (1-\nlb)\delta/d)
    \leq \pr(|N(v)\cap V_i|\le (1-\nlb)\deg(v)/d) \leq e^{\frac{-\nlb^2 \deg(v)}{2 d}} \leq e^{\frac{-\nlb^2 \delta}{2 d}}.
\]
By the union bound
we have
\[
  p:=\pr(\cA_v)\leq \sum_{i=1}^d \pr(\cA_{v,i}) \leq d e^{\frac{-\nlb^2 \delta}{2 d}}.
\]
Note that for every $u\in V$, there are at most $\Delta(\Delta-1)\le\Delta^2-1$ vertices $v$ for which $N(u)\cap N(v)\ne\es$.
Thus, the event $\cA_u$ is mutually independent of the set of all other events $\cA_v$
but at most $\Delta^2-1$.
To use the Lov\'asz Local Lemma~(\cref{thm:lll}), it is enough to show that
    $e p \Delta^2 \leq 1$.
Indeed,
letting $\delta=Cd\log{d}$,
\[
  ep\Delta^2
    \le e \cdot d e^{\frac{-\nlb^2 \delta}{2 d}} (\drat \delta)^2
    = e \drat^2 {\log^2}{d}\cdot C^2 d^{3-\frac{\nlb^2 C}{2}}.
\]
Write $\psi(d)={\log^2}{d}/d$.
The function $\psi(d)$ attains its maximum at $d=e^2$,
where it equals $4/e^2$.
Thus, for $C\ge 8/\nlb^2$, we obtain
\[
 e \drat^2 {\log^2}{d} C^2 d^{3-\frac{\nlb^2 C}{2}} = e \drat^2 \psi(d)
 C^2 d^{-(\nlb^2 C/2-4)}\leq \frac{4 \drat^2}{e}
 C^2 d^{-(\nlb^2 C/2-4)}\leq \frac{4 \drat^2}{e}
 C^2 2^{-(\nlb^2 C/2-4)}.
\]
Thus, there exists $C_1=C_1(\nlb,\drat)$ for which for every $C\ge C_1$
the above expression is at most $1$.
 
Therefore, by the Lov\'asz Local Lemma~(\cref{thm:lll}),
for $C\ge C_1$, the probability none of the $\cA_v$ occur is at least $(1-1/\Delta^2)^n$.
\end{proof}

In certain applications, we need the partition to be ``balanced''
with respect to some underlying colouring of the vertex set.
That is, given a colouring of the vertex set,
we need our partition to split each of the colour classes into roughly equally sized parts.
In our specific setting, we will exclusively apply this lemma to 
the case of two
colour classes (see \cref{prop:sparseconnector:is:bipartite:rigid}).
\begin{lemma}[Balanced partition lemma]\label{lem:partition:balanced}
  For every $\nlb,\psb\in(0,1)$ and $\drat,t \ge 1$ there exist $C=C(\nlb,\psb,\drat,t)>1$ 
  such that the following holds.
  Let $d\ge 2$, and let $G=(V,E)$ be an $n$-vertex graph with minimum degree
  $\delta\ge C d\log{d}$ and maximum degree $\Delta\le\drat\delta$.
  Assume that $V$ is partitioned into $t$ sets $A_1,\ldots,A_t$, such that $|A_i|=n/t$ for
  all $1\leq i\leq t$.
  Then, there exists a partition $V=V_1\cup\cdots\cup V_d$ such that every
  vertex in $V$ has at least $(1-\nlb)\delta/d$ neighbours in $V_i$ for all
  $1\leq i\leq d$, and $(1-\psb)n/(td)\leq |V_i\cap A_j|\leq (1+\psb)n/(td)$
  for all $1\leq i\leq d$ and $1\leq j\leq t$.
\end{lemma}

\begin{proof}
Let $\{m(v) : \, v\in V\}$ be a family of independent random variables, each
uniformly distributed on $[d]$. For all $1\leq i\leq d$, let
\[
    V_i= \{v\in V: m(v)=i\}.
\]
Let $\mathcal{A}$ be the event that all vertices in $V$ have at least
$(1-\alpha)n/d$ neighbours in each $V_i$, and let $\mathcal{B}$ be the event
that for all $1\leq i\leq d$ and $1\leq j\leq t$, $||V_i\cap A_j|-n/(td)|\leq
\psb n/(td)$. Let $p_1=\pr(\mathcal{A})$ and $p_2=\pr(\mathcal{B})$. In order
to show that $\pr(\mathcal{A}\cap \mathcal{B})>0$, it is enough to show that
$p_1+p_2>1$.
By the partition lemma (\cref{lem:partition}),
there exists $C_1=C_1(\alpha,\drat)$ such that, for $C\geq C_1$, one has
\[
 p_1 \geq \left(1-\frac{1}{\Delta^2}\right)^n
     \ge e^{-n/(\Delta^2-1)}
     \ge e^{-n/(\delta^2-1)}
     \ge e^{-2n/(C^2d)},
\]
where we used the inequality $1-x > e^{-x/(1-x)}$ which holds for $x\in(0,1)$,
and the fact that $\delta^2-1\geq C^2 d^2 {\log^2}{d}-1 \geq C^2 d -1 \geq C^2
d/2$.
On the other hand, for $1\leq i\leq d$ and $1\leq j\leq t$, we have by
Chernoff's bounds (\cref{thm:chernoff}),
\[
    \pr\left( \left| |V_i\cap A_j|- n/(td) \right| \geq \psb n/(td) \right) \leq 2e^{-\psb^2 n/(3td)},
\]
so by the union bound,
\[
    p_2 \geq 1- 2td e^{-\psb^2 n/(3td)}.
\]
Let $C_2=C_2(\psb,t)$ be such that $(\psb^2/(3t)) C^2 - (1+\log(2t)/\log 2 )C-2>0$
for all $C\geq C_2$. Note that $C_2> 6t/\psb^2$. Then, for $C\geq
\max\{C_1,C_2\}$, we obtain (using the inequality $n\geq \delta\geq C d\log
d$),
\begin{align*}
\log(p_1/(1-p_2)) &\geq - \frac{2n}{C^2 d}+\frac{\psb^2n}{3td}-\log(2 t d)
=\frac{n}{d} \left(\frac{\psb^2}{3t}-\frac{2}{C^2}\right)-\log(2td)
\\
&\geq
C\log d \left(\frac{\psb^2}{3t}-\frac{2}{C^2}\right)-\log(2td)
=\frac{\log d}{C} \left(\frac{\psb^2}{3t} C^2 -\frac{\log(2 t d)}{\log d}C -2\right).
\end{align*}
Since
\[
    \frac{\log(2td)}{\log d} = 1+\log(2t)/\log d \leq 1+\log(2t)/\log 2,
\]
we obtain
\[
\log(p_1/(1-p_2)) \geq \frac{\log d}{C} \left(\frac{\psb^2}{3t} C^2
-\left(1+\frac{\log(2t)}{\log 2}\right)C -2\right)>0.
\]
Thus, for $C\geq \max\{C_1,C_2\}$, we obtain $p_1+p_2>1$, as wanted.
\end{proof}

We now embark on the task of identifying basic graph properties that will
guarantee the existence of rigid partitions.
Say that a graph $G=(V,E)$
is an \defn{$R$-expander} if every set $A\subseteq V$ with $|A|\le R$ has $|N(A)|\ge 2|A|$.
We say it
is \defn{$(x,y)$-sparse} if
every set $A\subseteq V$
with $|A|=a\le x$
spans at most $ay$ edges.
Note that if $G$ is $(x,y)$-sparse
then it is $(x',y')$-sparse
for every $x'\le x$ and $y'\ge y$.
Moreover, if $G$ is $(x,y)$-sparse
then every subgraph $G'$ of $G$ is $(x,y)$-sparse.
The next lemma roughly shows that sparse graphs
with sufficiently high minimum degree
are expanders.

\begin{lemma}\label{lemma:sparse_implies_expander}
    Let $G=(V,E)$ be an $(x,y)$-sparse graph with minimum degree $\delta\geq 6y$. Then, $G$ is an $(x/3)$-expander.
\end{lemma}
\begin{proof}
  Let $A\subseteq V$ with $|A|\le x/3$,
  and let $B=A\cup N(A)$.
  Assume for contradiction that
  $|N(A)|< 2|A|$. 
  Then $|B|<3|A|\le x$,
  and therefore, since $G$ is $(x,y)$-sparse,
  we have
  \[
    |E(B)|\le |B|y < 3|A|y.
  \]
  On the other hand,
  since $\delta(G)=\delta\ge 6y$,
  we have
  \[
    |E(B)|\ge \frac{1}{2}\sum_{v\in A} \deg(v) 
    \ge \delta|A|/2
    \ge 3|A|y,
  \]
    a contradiction.
\end{proof}

Say that a graph $G=(V,E)$
is a \defn{$K$-connector} if every two disjoint vertex sets of size at least $K$
are connected by an edge. Similarly, if $G=(V_1\cup V_2,E)$ is a bipartite graph, we say that $G$ is a \defn{$K$-bi-connector} if  every $A_1\subseteq V_1$ and $A_2\subseteq V_2$,
both of size at least $K$, are connected by an edge. Note that if $G$ is a $K$-connector, then it is also a $K'$-connector for any $K'>K$. Moreover, if $G$ is a $K$-connector, then, for every $U\subseteq V$, $G[U]$ is also a $K$-connector, and for every two disjoint sets $U_1,U_2\subseteq V$, $G[U_1,U_2]$ is a $K$-bi-connector. 
Finally, note that the property of being a $K$-connector is monotone increasing.
The next lemma roughly shows that
graphs that are both expanders and connectors
are connected.

\begin{lemma}\label{lemma:expander+connector}
  Let $G$ be an $R$-expander. Then, if $G$ is a $3R$-connector, $G$ is connected.
  Similarly, if $G$ is bipartite and a $2R$-bi-connector, then $G$ is connected.
\end{lemma}

\begin{proof}
  Note that for any $R$-expander $G$, every connected component of $G$ is of size at least $3R$.
  On the other hand, if $G$ is a $K$-connector,
  any connected component except possibly the largest one must have less than $K$ vertices.
  Therefore, any $R$-expander that is also a $3R$-connector must be connected. 

  Now, assume that $G$ is bipartite with parts $V_1$ and $V_2$.
  Note that if $G$ is an $R$-expander,
  then every connected component of $G$ has at least $2R$ elements in each $V_i$.
  On the other hand, if $G$ is a $K$-bi-connector,
  any connected component expect possibly the largest one has less than $K$ elements in one of $V_1,V_2$.
  Thus, a bipartite $R$-expander that is also a $(2R)$-bi-connector is connected.
\end{proof}

Using the last two lemmas,
we now show that in sparse connectors,
induced subgraphs with sufficiently large
minimum degree are connected.
\begin{lemma}\label{lemma:sparse+connector+minimum_degree}
  Let $G=(V,E)$ be an $(x,y)$-sparse graph
  that is also a $(2x/3)$-connector.
  Let $U_1,U_2\subseteq V$
  such that either $U_1=U_2$
  or $U_1\cap U_2=\es$,
  and assume that $\delta(G[U_1,U_2])\ge 6y$.
  Then, $G[U_1,U_2]$ is connected.
\end{lemma}
\begin{proof}
    Let $G'=G[U_1,U_2]$. Note that any subgraph of an $(x,y)$-sparse graph is
    also $(x,y)$-sparse. In particular, $G'$ is $(x,y)$-sparse. Since the
    minimum degree of $G'$ is at least $6y$,
    then by \cref{lemma:sparse_implies_expander} $G'$ is an $(x/3)$-expander. 

    We now divide into two cases:
    if $U_1=U_2$ then, since $G$ is a $(2x/3)$-connector, so is $G'=G[U_1]$. In particular, $G'$ is an $x$-connector,
    so by \cref{lemma:expander+connector} it is connected.
    If $U_1\cap U_2=\es$, then, since $G$ is a $(2x/3)$-connector, $G'=G[U_1,U_2]$ is a $(2x/3)$-bi-connector.
    Thus, by \cref{lemma:expander+connector}, it is connected.
\end{proof}

We now deduce from the last lemma
and the partition lemma (\cref{lem:partition})
a useful sufficient condition
for the admittance of {\em strong}
rigid partitions.
\begin{proposition}\label{prop:sparseconnector:is:rigid}
  For every $\drat\geq 1$ there exists $C=C(\drat)>1$ such that the following holds.
  Let $d\geq 2$ and let $G=(V,E)$ be a graph
  with minimum degree $\delta\geq C d \log d$ and maximum degree $\Delta\leq \drat \delta$.
  Assume that $G$ is $(x,\delta/(7d))$-sparse and a $(2x/3)$-connector, for some $x$.
  Then $G$ admits a strong $d$-rigid partition.
\end{proposition}

\begin{proof}
    Let $\nlb=1/7$. By the partition lemma (\cref{lem:partition}),
    there is $C=C(\nlb,\drat)=C(\drat)$
    such that if $\delta\geq C d \log d$ and $\Delta\leq \drat \delta$,
    there exists a partition $V=V_1\cup\cdots \cup V_d$
    such that the minimum degree of  $G[V_i,V_j]$
    is at least $(1-\nlb)\delta/d=6\delta/(7d)$ for all $1\leq i,j\leq d$.
    By \cref{lemma:sparse+connector+minimum_degree},
    $G[V_i,V_j]$ is connected for all $1\leq i,j\leq d$, as wanted.
\end{proof}

In order to address the bipartite scenario,
we will utilize the following simple lemma.

\begin{lemma}\label{lem:paths_DFS}
  Let $m,K \geq 1$ be integers,
  and let $G$ be a bipartite graph with sides $A,B$ of size at least $m$ each.
  Suppose further that $G$ is $K$-bi-connector.
  Then, $G$ admits a matching of size $m-K+1$.
\end{lemma}
\begin{proof}
    Let $M$ be a maximal matching in $G$. Assume for contradiction that $|M|\leq m-K$. Let $A'\subseteq A$ and $B'\subseteq B$ be the vertices not incident to any edge of $M$. Then, we have $|A'|=|B'|\geq K$. Since $G$ is a $K$-bi-connector, there exists an edge in $G$ connecting a vertex $u\in A'$ and a vertex $v\in B'$. But then $M\cup\{\{u,v\}\}$ is a matching containing $M$, in contradiction to the maximality of $M$.
\end{proof}

\begin{proposition}\label{prop:sparseconnector:is:bipartite:rigid}
  For every $\drat\geq 1$ there exists $C=C(\drat)>1$ such that the following holds.
  Let $d\geq 2$, write $d'=d+1$, and let $G=(V,E)$ be a bipartite graph with bipartition $(A,B)$
  with $|A|=|B|$,
  minimum degree $\delta\geq C d' \log d'$ and maximum degree $\Delta\leq \drat \delta$.
  Assume that $d' \le \sqrt{3|A|/4}$,
  and that $G$ is $(x,\delta/(7d'))$-sparse and a $(2x/3)$-bi-connector
  for $x\le|A|/(7d')$.
  Then $G$ admits a strong bipartite $d$-rigid partition.
\end{proposition}

\begin{proof}
  By the balanced partition lemma (\cref{lem:partition:balanced}),
  there is $C=C(\drat)$
  such that if $\delta \ge Cd'\log{d'}$
  and $\Delta \le \drat\delta$,
  there exists a partition $V=V_1\cup\dots \cup V_d\cup V_{d+1}$
  with the following properties.
  Write $A_i=V_i\cap A$ and $B_i=V_i\cap B$.
  Then, $\delta(G[V_i,V_j])\ge \frac{6\delta}{7d'}$ for all $1\le i<j\le d+1$,
  and hence $\delta(G[A_i,B_j])\ge \frac{6\delta}{7d'}$ for all $i\ne j$ in $[d+1]$.
  In addition,
  $|A_i|,|B_i| \ge \frac{6|A|}{7d'} $ for all $1\le i\le d+1$.
  Fix $i\ne j$ and let $G'=G[A_i,B_j]$.
  Repeating the argument in the proof of \cref{lemma:sparse+connector+minimum_degree},
  we observe that by \cref{lemma:sparse_implies_expander}, $G'$ is an $(x/3)$-expander;
  and by \cref{lemma:expander+connector}, it is connected.
  Let $G^*=G[A_i,B_i]$.
  Since $G^*$ is $(2x/3)$-bi-connector,
  it follows from \cref{lem:paths_DFS} that it contains a matching of size
  at least $\frac{6|A|}{7d'}-\frac{2x}{3} > \frac{3|A|}{4d'} > d$.
  Thus, by \cref{lem:matching:bipartite},
  $G$ admits a strong bipartite $d$-rigid partition.
\end{proof}

\section{Pseudorandom and random graphs}\label{sec:randomgraphs}
\subsection{Jumbled graphs}\label{sec:jumbled}
Roughly speaking,
pseudorandom graphs are graphs with properties that one expects to find in random graphs.
One such property, captured by the notion of ``jumbledness'',
originally introduced by Thomason~\cites{Tho87a,Tho87b}\footnote{%
Here, we do not use the original definition of Thomason,
but rather a similar definition that appears in~\cite{ConPG}.}
is that the graph has essentially the same edge density between any two (sufficiently large) vertex sets.
Formally,
we say that a graph $G=(V,E)$ is \defn{$(p,\beta)$-jumbled}
if it satisfies,
for all $A,B\subseteq V$,
\[
  \left|e(A,B)- p|A||B|\right|
  \le \beta \sqrt{|A||B|}.
\]
In the next two lemmas
we show that jumbled graphs are sparse connectors.
\begin{lemma}\label{lem:jumbled:is:sparse}
  A $(p,\beta)$-jumbled graph is $(x,(px+\beta)/2)$-sparse
  for every positive integer $x$.
\end{lemma}

\begin{proof}
  Let $x$ be a positive integer,
  and let $A\subseteq V$ with $|A|\le x$.
  Then,
  \[
    2|E(A)|=e(A,A)
    \le p|A|^2+\beta|A| \le |A|(px+\beta).\qedhere
  \]
\end{proof}

\begin{lemma}\label{lem:jumbled:is:connector}
  A $(p,\beta)$-jumbled graph is $K$-connector
  for every $K>\beta/p$.
\end{lemma}

\begin{proof}
  Let $A,B\subseteq V$ with $A\cap B=\es$
  and $|A|=|B|=K>\beta/p$.
  Then,
  \[
    |E(A,B)|=e(A,B)\ge p|A||B|-\beta\sqrt{|A||B|}
    = K(pK-\beta)>0.\qedhere
  \]
\end{proof}

\begin{proposition}\label{prop:jumbled:is:rigid}
  For every $\drat$ there exists $C=C(\drat)>1$ such that
  the following holds.
  Let ${d\ge 2}$,
  and let $G=(V,E)$ be a $(p,\beta)$-jumbled graph
  with minimum degree $\delta\ge \max\{9d \beta,Cd\log{d}\}$ and
  maximum degree $\Delta\le \drat\delta$.
  Then $G$ admits a strong $d$-rigid partition.
\end{proposition}

\begin{proof}
   Let $x=11\beta/(7p)$.
   Note that $\delta/(7d)\geq 9\beta/7 = (px+\beta)/2$. Hence, by \cref{lem:jumbled:is:sparse},
  $G$ is $(x,\delta/(7d))$-sparse.
  Since $2x/3 > \beta/p$, by \cref{lem:jumbled:is:connector}
  $G$ is a $(2x/3)$-connector.
Let $C=C(\drat)$ be the constant from \cref{prop:sparseconnector:is:rigid}. Since $\delta\geq C d \log d$ and $\Delta\leq \drat\delta$, by \cref{prop:sparseconnector:is:rigid}, $G$ admits a strong $d$-rigid partition.
\end{proof}

To prove \cref{thm:nkl}, we use the Expander Mixing Lemma, stated below.

\begin{theorem}[\cite{AC88}; see also \cite{HLW06}*{Lemma~2.5}]\label{thm:xml}
  Every $(n,k,\lambda)$-graph
  is $(k/n,\lambda)$-jumbled.
\end{theorem}

\Cref{thm:nkl} follows immediately:
\begin{proof}[Proof of \cref{thm:nkl}]
  Let $G$ be an $(n,k,\lambda)$-graph.
  By the Expander Mixing Lemma (\cref{thm:xml})
  $G$ is $(k/n,\lambda)$-jumbled.
  The result follows from \cref{prop:jumbled:is:rigid} by taking $C=C(1)$.
\end{proof}

\subsection{Spread measures}
Following (and slightly extending) terminology from~\cite{ALWZ21},
we say that a distribution $\cG$ of $n$-vertex graphs is \defn{$(p;Z)$-spread}
if for every set $E_0$ of pairs of vertices with $|E_0|=m\le Z$
we have that $\pr(E_0\subseteq E(\cG))\le p^m$.
We say that $\cG$ is \defn{$p$-spread} if it is $(p;\binom{n}{2})$-spread.
This section is dedicated to proving that spread measures of sufficiently high density
are (locally and relatively) ``sparse''.
The following lemma,
which will be instrumental in the next few sections,
makes this concept precise.

\begin{lemma}\label{lem:spread:sparse}
  For every $\Lambda>0$ there exists $k>0$ such that the following holds.
  Suppose $p\ge k/n$,
  and let $\cG$ be a $(p;Z)$-spread distribution of $n$-vertex graphs.
  Set 
  $x=2\Lambda e^{-(1+1/\Lambda )}\log(np)/p$
  and $y=\Lambda \log(np)$,
  and assume that $xy\le Z$.
  Then,
  $G\sim\cG$ is \whp{}
  $(x,y)$-sparse
\end{lemma}

\begin{proof}
  Let $\eta=2\Lambda  e^{-(1+1/\Lambda )}$.
  For an edge set $F\subseteq E(K_n)$ of size $m\le Z$,
  the probability that $F\subseteq E(G)$ is at most $p^m$.
  Fix $1\le a\le x$
  and let $m=ay\le xy\le Z$.
  By the union bound,
  the probability that there exists a vertex set of size $a$
  that spans at least $m$ edges
  is at most
  \[
    \binom{n}{a}\binom{\binom{a}{2}}{m}p^m
    \le \left(\frac{en}{a}\cdot\left(\frac{ea p}{2y}\right)^y\right)^a.
  \]
  Write $f(a)=(en/a)\cdot(eap/(2y))^y$ and
  note that $f$ is increasing in $a$.
  Thus, for every $a\le x$, taking $k>\exp(2e/\eta)$,
  \[\begin{aligned}
    f(a)&\le f(x)
    = \frac{en}{x}\cdot\left(\frac{exp}{2y}\right)^y
    = \frac{enp}{\eta\log(np)}\cdot\left(\frac{e\eta}{2\Lambda}\right)^{\Lambda\log(np)}\\
    &
    = \frac{enp}{\eta\log(np)}\cdot\left(e^{-1/\Lambda}\right)^{\Lambda\log(np)}
    =
    \frac{enp }{\eta\log(np)}\cdot (np)^{-1}
    \le \frac{e}{\eta\log{k}} < \frac{1}{2}.
  \end{aligned}\]
  Let $x'=x'(n)$ satisfy $1\ll x' \ll x$.
  Note that $f(x')/f(x)=(x'/x)^{y-1}\ll 1$,
  hence, for $a\le x'$, $f(a)\le f(x')\ll f(x)< 1/2$.
  Then, by the union bound, the probability that there exists a set of size at most $x$
  that spans at least $ay$ edges
  is at most
  \[
    \sum_{a=1}^{x} (f(a))^a
    \le\sum_{a=1}^{x'} (f(a))^a
    +\sum_{a=x'}^{x} (f(a))^a
    \le\sum_{a=1}^{\infty} (f(x'))^a
    +\sum_{a=x'}^{\infty} (f(x))^a
    = o(1),
  \]
  as required.
\end{proof}

\subsection{Random regular graphs}\label{sec:gnr}
\newcommand{\cfg}{\cM}
\newcommand{\cfgm}{\cG^*}
To prove \cref{thm:gnr}, we show that random regular graphs are \whp{} sparse connectors,
and apply \cref{prop:sparseconnector:is:rigid}.
We cast our arguments in the {\em configuration model}.
Let us lay formal grounds.
Let $K_{n\cdot k}$ be the complete graph on the vertex set $[n]\times[k]$.
The vertices of $K_{n\cdot k}$ are called \defn{half-edges},
and the vertex sets $V_i=\{(i,j):\ j\in [k]\}$, $i\in[n]$, are called \defn{clusters}.
For a subgraph $G\subseteq K_{n\cdot k}$, denote by $\pi(G)$ the multigraph obtained from $G$
by contracting each cluster $V_i$ into a single vertex $i$.
Assume that $nk$ is even,
and denote by $\cfg_{n,k}$ the uniform distribution over perfect matchings in $K_{n\cdot k}$.
Denote further by $\cfgm_{n,k}$ the distribution of $\pi(M)$ for $M\sim\cfg_{n,k}$.
We will repeatedly use the following  result (see, e.g.,~\cite{JLR}*{Theorem~9.9})):
if $\cfgm_{n,k}$ satisfies a property \whp{}
then $\cG_{n,k}$ satisfies that property \whp{} as well.

To prove that random regular graphs are typically sparse,
we show that they (or, rather, $\cfgm_{n,k}$) are typically well-spread.
\begin{lemma}\label{lem:gnr:edges}
  For every $\eps\in(0,1)$ and every $k\ge 3$,
  $\cfgm_{n,k}$ is $(p;M)$-spread
  for $p=\eps^{-1}k/n$ and $M=(1-\eps)kn/2$.
\end{lemma}

\begin{proof}
  Let $G^*\sim\cfgm_{n,k}$,
  and let $E_0$ be a set of pairs of vertices of $G^*$ of size at most $(1-\eps)kn/2$.
  We think of $E_0$ as a (simple) subgraph of $K_n$.
  Note that the number of matchings $M_0$ of size $m$ of $K_{n\cdot k}$ for which $\pi(M_0)=E_0$
  is at most $(k^2)^m$.
  Fix one such matching $M_0=\{e_1,\dots,e_m\}$.
  Let $M\sim\cfg_{n,k}$, and note that for every $i\in[m]$,
  \[
    \pr(e_i\in M\mid e_1,\dots,e_{i-1}\in M)
    \le \frac{1}{kn-2m}
    \le \frac{1}{\eps kn}.
  \]
  Thus, $\pr(M_0\subseteq M)\le(\eps kn)^{-m}$.
  We couple $G^*\sim\cfgm_{n,k}$ with $M$ by letting $G^*=\pi(M)$.
  By the union bound,
  $\pr(E_0\subseteq E(G^*)) \le (\eps^{-1}k/n)^m$.
\end{proof}

\begin{lemma}\label{lem:gnr:sparse}
  For every sufficiently large constant $k$,
  $G\sim\cG_{n,k}$ is \whp{} $(x,y)$-sparse for $x=12\log{k}\cdot n/k$ and $y=70\log{k}$.
\end{lemma}

\begin{proof}
  We prove the statement for $G\sim\cfgm_{n,k}$ and deduce it for $\cG_{n,k}$.
  Let $\Lambda=35$.
  By \cref{lem:gnr:edges}, for $\eps=1/2$,
  $\cfgm_{n,k}$ is $(p;M)$-spread for $p=2k/n$ and $M=kn/4$.
  Note that
  $x=12\log{k}\cdot n/k \le 24\log(np)/p \le 2\Lambda e^{-(1+1/\Lambda)} \log(np)/p$,
  and $y=70\log{k}\ge\Lambda\log(np)$.
  Note also that $xy<M$ for sufficiently large $k$.
  Thus, by \cref{lem:spread:sparse}, for sufficiently large $k$,
  $G^*\sim\cfgm_{n,k}$ is \whp{} $(x,y)$-sparse.
\end{proof}

We proceed to prove that $\cG_{n,k}$ is typically a good connector.
\begin{lemma}\label{lem:gnr:connect}
  Suppose $k\ge 3$,
  and let $G\sim \cG_{n,k}$.
  Then, \whp{}
  $G$ is a $K$-connector for  $K=8\log{k}\cdot\frac{n}{k}$.
\end{lemma}

\begin{proof}
  As before, we prove the statement for $G\sim\cfgm_{n,k}$ and deduce it for $\cG_{n,k}$.
  Let $U,W$ be disjoint vertex sets of $K_{n\cdot k}$ with $|U|=a$ and $|W|=b$.
  Write $U=\{u_1,\dots,u_a\}$.
  For $i\in[a]$ let $\cA_i$ be the event
  that $M\sim\cfg_{n,k}$ does not contain an edge from $u_i$ to $W$.
  Denote by $\cB_i$ the event that $M$ contains an edge from $u_i$ to $u_j$ for $j<i$.
  Note that
  \[
    \pr\left(\cA_i\mid\neg\cB_i\wedge\left(\bigwedge_{j<i}\cA_j\right)\right)
    \le 1-\frac{b}{kn-1}.
  \]
  Thus, since $\cB_i$ --- deterministically --- occurs for at most $a/2$ indices $i$,
  \[
    \pr\left(\bigwedge_i \cA_i\right) \le \left(1-\frac{b}{kn-1}\right)^{a/2}.
  \]
  Now, we couple $G$ with $M$ by setting $G=\pi(M)$.
  We deduce from the discussion above and the union bound that
  the probability that there exist two vertex sets $A,B$ of $G$ with $|A|=|B|=K$
  and $E_G(A,B)=\es$ is at most
  \[\begin{aligned}
    \binom{n}{K}^2\left(1-\frac{kK}{kn-1}\right)^{kK/2}
    &\le \left(\frac{en}{K}\right)^{2K}\exp\left(-\frac{kK^2}{4n}\right)\\
    &\le \left(\frac{e^2 k^2}{8^2{\log^2}{k}}\cdot\exp\left(-2\log{k}\right)\right)^K
    = \left(\frac{1}{8{\log^2}{k}}\right)^K = o(1).
  \end{aligned}\]
  Thus, $G$ is \whp{} a $K$-connector.
\end{proof}

\begin{proof}[Proof of \cref{thm:gnr}]
  Let $C_1=C(1)$ be the constant from \cref{prop:sparseconnector:is:rigid},
  and let $\Lambda=70$.
  Let $C$ be a constant satisfying $C\geq C_1$ and $C\geq 28\Lambda\log{C}$.
  Let $k\geq Cd\log{d}$.
  We also assume that $k$ is sufficiently large for \cref{lem:gnr:sparse} to hold
  (otherwise, we make $C$ larger).
  Let $x=12\log{k}\cdot n/k$ and $y=70\log{k}$.
  By \cref{lem:gnr:sparse}, $G$ is \whp{} $(x,y)$-sparse,
  and by \cref{lem:gnr:connect}, $G$ is \whp{} a $(2x/3)$-connector.
  Note that
    \begin{align*}
    y &= \Lambda \log{k}
       = \Lambda\log(C d\log{d})
       \leq \Lambda \log{C}+2\Lambda\log{d}\\
      &\leq  C/28+C\log{d}/14
       \leq C\log{d}/14+C\log{d}/14
       = C\log{d}/7
       \leq \frac{k}{7d}.
    \end{align*}
  So $G$ is \whp{} $(x,k/(7d))$-sparse.
  Hence,
  since $k\ge C d\log{d}\ge C_1d\log{d}$,
  by \cref{prop:sparseconnector:is:rigid}, $G$ \whp{} admits a $d$-rigid partition.
\end{proof}

\subsection{Binomial random graphs}\label{sec:gnp}
We begin with several lemmas about expansion properties of $G(n,p)$.
The following is a well-known statement about the degrees of $G(n,p)$.
\begin{lemma}\label{lem:gnp:deg}
  For every $\eps>0$ there exists $\xi>0$ such that the following holds.
  Let $p\ge (1+\eps)\log{n}/n$, and let $G\sim G(n,p)$.
  Then, \whp{}, $(1-\xi)np\le\delta(G)\le\Delta(G)\le enp$.
\end{lemma}

\begin{proof}
  The degree of a vertex is distributed as a binomial random variable with $n-1$ attempts
  and success probability $p$.
  Thus, by Chernoff bounds (\cref{thm:chernoff}),
  \begin{align*}
    \pr(\deg(v)\ge enp)
      &\le \exp(-np\phi(e-1)) \le n^{-(1+\eps)} \ll n^{-1},\\
    \pr(\deg(v)\le (1-\xi)np)
      &\le \exp\left(-np\phi(-\xi)\right) \le n^{-(1+\eps)\phi(-\xi)},
  \end{align*}
  where the last expression is $o(n^{-1})$ if $\xi<1$ is close enough to $1$
  (since $\phi(x)\to 1$ as $x\to -1$).
  The result follows by applying the union bound over all $n$ vertices.
\end{proof}

Since $G(n,p)$ is (trivially) $p$-spread, we obtain the following.

\begin{lemma}\label{lem:gnp:sparse}
  There exists a sufficiently large constant $k$ such that if $p\ge k/n$, then
  $G\sim G(n,p)$ is \whp{}
  $(x,y)$-sparse
  for
  (i)
  $x=3\log(np)/p$
  and $y=5\log(np)$;
  (ii)
  for $x=7\log(np)/p$
  and $y=11\log(np)$.
\end{lemma}

\begin{proof}
  The result follows from \cref{lem:spread:sparse}
  since $G(n,p)$ is $p$-spread,
  and by noting that
  $3\le 2\cdot 5 e^{-(1+1/5)}$
  and 
  $7\le 2\cdot 11 e^{-(1+1/11)}$.
\end{proof}

\begin{lemma}\label{lem:gnp:connect}
  Suppose $p\ge 4/n$,
  and let $G\sim G(n,p)$.
  Then, \whp{}
  $G$ is a $K$-connector for  $K= 2\log(np)/p$.
\end{lemma}

\begin{proof}
  Let $A,B$ be two disjoint vertex sets with $|A|=|B|=K$.
  The probability that $G$ contains no edge between $A$ and $B$ is $(1-p)^{K^2}$.
  Thus, by the union bound over all sets $A,B$,
  the probability that the assertion of the claim fails is at most
  \[
    \binom{n}{K}^2 (1-p)^{K^2}
    \le \left(\left(\frac{en}{K}\right)^2 e^{-pK}\right)^{K}
    = \left(\frac{e}{2\log(np)}\right)^{2K}
    < 0.99^{2K}
    = o(1).\qedhere
  \]
\end{proof}

We are now ready to prove \cref{thm:gnp}.

\begin{proof}[Proof of \cref{thm:gnp}]
  Let $\xi=\xi(\eps)$ be the constant guaranteed by \cref{lem:gnp:deg}, and let
   $\drat=e/(1-\xi)$.
  Write $\delta=\delta(G)$ and $\Delta=\Delta(G)$.
  Note that by \cref{lem:gnp:deg}, $\Delta\le \drat\delta$  \whp{}.
  Let $C=C(\drat)$ be the constant from \cref{prop:sparseconnector:is:rigid}.
  Let $c=\min\{1/35, 1/C\}$, and let
  $d=c \delta/\log(np)$.
  
  Let $x=3\log(np)/p$ and $y=5\log(np)$.
  By \cref{lem:gnp:sparse},
  $G$ is \whp{} $(x,y)$-sparse.
  By \cref{lem:gnp:connect}, $G$ is \whp{} a $(2x/3)$-connector.
  Note that $5 \log(np)\leq\delta/(7d)$, so
  $G$ is \whp{} $(x,\delta/(7d))$-sparse.
  Furthermore, using the fact that $d\leq c\delta\leq c e n p < np$, we obtain $\delta = d \log(np)/c \geq C d \log(np)> C d \log{d}$.
  Hence, by \cref{prop:sparseconnector:is:rigid}, $G$ \whp{} admits a $d$-rigid partition.
\end{proof}

\subsection{Random graph processes}\label{sec:gnt}
In this section we prove \cref{thm:gnt}.
Let $2\le d=d(n)\le c\log{n}/\log\log{n}$
for a small constant $c>0$ to be determined later,
and let $G\sim G(n,\tau_d)$.
For $\eta>0$ denote $S_\eta=\{v\in V: \deg_G(v)<\eta\log{n}\}$.
It will be useful to couple $G$ with
$G^+\sim G(n,p)$ for $p=2\log{n}/n$,
$G^-\sim G(n,p)$ for $p=\log{n}/n$,
and $G^\circ\sim G(n,m)$ for $m=n\log{n}$,
so that $G^-$ is a subgraph of $G$
and $G$ is a subgraph of $G^+,G^\circ$.
Such a coupling is possible since $\tau_d<n\log{n}$ \whp{} (see \cref{lem:gnp:deg})
and by using standard translations between $G(n,p)$ and $G(n,m)$ (see, e.g.,~\cite{FK}).

\begin{claim}\label{cl:small}
  There exists $\eta>0$
  such that in $G$,
  \whp{},
  the distance between any two vertices in $S_\eta$ is at least $5$,
  and no vertex in $S_\eta$ lies in a cycle of length at most $4$.
\end{claim}

\begin{proof}
  From monotonicity, it suffices to prove that \whp{}
  $G^\circ$ does not contain a path of length $4$
  between two potentially identical vertices in $S_\eta$.
  This fact is proved in~\cite{Kri16}*{Lemma~4.5} for a small enough constant $\eta>0$.
\end{proof}

\begin{proof}[Proof of \cref{thm:gnt}]
  For $d=1$ the claim follows from the known fact that $G(n,\tau_1)$ is \whp{} connected
  (see, e.g.,~\cite{ER61}).
  Therefore, we will assume $d\ge 2$.
  We also assume that $d\le c\log{n}/\log\log{n}$
  for $c>0$ to be chosen later.
  Let $\eta>0$ be the constant from \cref{cl:small}.
  Let $V^*=V\sm S_\eta$
  and consider the graph $G^*=G[V^*]$.
  Write $\delta=\delta(G^*)$
  and $\Delta=\Delta(G^*)$.
  Note that by \cref{lem:gnp:deg},
  $\Delta\le \drat\delta$ for $\drat=e/\eta$ \whp{}.
  Let $\nlb=1/8$
  and let $C=C(\nlb,\drat)$ be the constant guaranteed by the partition lemma (\cref{lem:partition}).
  Let $c=\min\{\eta/12, \eta/C\}$.
  Note that
  $\delta\ge\eta\log{n}\ge Cc\log{n} \ge Cd\log\log{n} > Cd\log{d}$.
  Thus, there exists a partition $V^*=V_1\cup\dots\cup V_d$
  with $e(v,V_i)\ge 7\delta/(8d)$ for every $v\in V^*$ and $i\in[d]$.

  We now adjust the partition as follows:
  for every $v\in S_\eta$, let $v_1,\dots,v_d$ be $d$ arbitrary neighbours of $v$.
  Note that \cref{cl:small} implies that the sets $\{v_1,\dots,v_d\}$
  are \whp{} disjoint for distinct $v\in S_\eta$,
  and are subsets of $V^*$.
  For every $v\in S_\eta$ and $i\in[d]$, we move $v_i$ to $V_i$.
  Note that after the change,
  $e(v,V_i)\ge 7\delta/(8d)-1 \ge 6\delta/(7d)$ (for large enough $n$)
  for every $v\in V^*$ and $i\in[d]$.

  Let $p=\log{n}/n$ and $x=3\log(np)/p$.
  By \cref{lem:gnp:connect},
  $G^-$ (and hence $G$, and hence $G^*$) is \whp{} a $(2x/3)$-connector.
  Let $p'=2p$, and note that $x\sim 6\log(np')/p'<7\log(np')/p'$.
  Let $y=11\log(np')\sim 11\log\log{n}$.
  By \cref{lem:gnp:sparse},
  $G^+$ (and hence $G$, and hence $G^*$) is \whp{} $(x,y)$-sparse.
  Noting that
  $6\delta/(7d) \ge \frac{6\eta}{c}\log\log{n} \ge 72\log\log{n} > 6y$,
  by \cref{lemma:sparse+connector+minimum_degree},
  $G[V_i,V_j]$ is connected for all $1\leq i,j\leq d$,
  and thus the partition $V^*=V_1\cup\dots\cup V_d$ is a strong $d$-rigid partition of $G^*$.
  By adding $S_\eta$ to $V_1$ (say),
  noting that every $v\in S_\eta$ has a neighbour in each $V_i$,
  we obtain a strong $d$-rigid partition of $G$.
\end{proof}

\subsection{Binomial random bipartite graphs}\label{sec:gnnp}
We state the following lemmas without proofs;
they are the bipartite analogs of \cref{lem:gnp:deg,lem:gnp:sparse,lem:gnp:connect},
and their proofs are essentially identical.

\begin{lemma}\label{lem:gnnp:deg}
  For every $\eps>0$ there exists $\xi>0$ such that the following holds. 
Let $p\geq (1+\eps)\log{n}/n$, and let $G\sim G(n,n,p)$. Then, \whp{}, $(1-\xi)np\le\delta(G)\le\Delta(G)\le enp$.
\end{lemma}

\begin{lemma}\label{lem:gnnp:sparse}
  There exists a sufficiently large constant $k$ such that if $p\ge k/n$, 
  $G\sim G(n,n,p)$ is \whp{}
  $(x,y)$-sparse
  for $x=3\log(np)/p$
  and $y=5\log(np)$.
\end{lemma}

\begin{lemma}\label{lem:gnnp:connect}
  Suppose $p\gg 1/n$
  and let $G\sim G(n,n,p)$.
  Then, \whp{}
  $G$ is a $K$-bi-connector for  $K= 2\log(np)/p$.
\end{lemma}

\begin{proof}[Proof of \cref{thm:gnnp}]
  Let $\xi=\xi(\eps)$ be the constant guaranteed by \cref{lem:gnnp:deg}, and let
   $\drat=e/(1-\xi)$.
  Write $\delta=\delta(G)$ and $\Delta=\Delta(G)$.
  Note that by \cref{lem:gnnp:deg}, $\Delta\le \drat\delta$  \whp{}.
  Let $C=C(\drat)$ be the constant from \cref{prop:sparseconnector:is:bipartite:rigid}.
  Let $c=\min\{1/70, 1/(2C)\}$,
  and let $d=\min\{c\delta/\log(np),0.8\sqrt{n}\}$.
  
  We assume that $\delta\le enp$, as guaranteed to occur \whp{} by \cref{lem:gnnp:deg}.
  Let $d'=d+1$,
  $x=3\log(np)/p$ and $y=5\log(np)$.
  Note that
  $x = 3\log(np)/p \le 3en\log(np)/\delta = 3cen/d<n/(7d')$.
  By \cref{lem:gnnp:connect}, $G$ is \whp{} a $(2x/3)$-bi-connector.
  By \cref{lem:gnnp:sparse},
  $G$ is \whp{} $(x,y)$-sparse.
  Note that $5 \log(np)\leq\delta/(7d')$, so
  $G$ is \whp{} $(x,\delta/(7d'))$-sparse.
  Furthermore, using the fact that $d\leq c\delta\leq c e n p \le np-1$,
  we obtain $\delta \ge d \log(np)/c \geq 2C d \log(np)> C d' \log(d')$.
  Finally, $d'\le 0.8\sqrt{n}+1 <\sqrt{3n/4}$.
  Hence, by \cref{prop:sparseconnector:is:bipartite:rigid}, $G$ \whp{} admits a strong bipartite $d$-rigid partition.
\end{proof}

\subsection{Giant rigid component}\label{sec:giant}
In this section we prove \cref{thm:rigid_component},
showing the typical existence of a giant $d$-rigid component in sparse binomial random graphs.
The initial step in this direction is the establishment of a lemma,
which posits the presence of a giant ``balanced'' component.
\begin{lemma}\label{lem:giant:balanced}
  For any $\eps\in(0,1)$
  and $c\in(0,1-\eps)$
  there exists $k_0=k_0(\eps,c)>0$ such that for every $k\geq k_0$,
  $G(n,k/n)$ contains \whp{}
  an induced subgraph on at least $(1-\eps)n$ vertices,
  with minimum degree at least $ck$
  and maximum degree at most $4k/\eps$.
\end{lemma}

For the proof of \cref{lem:giant:balanced}
we use the following
generalisation of~\cite{AHHK22+}*{Lemma~5.2}.
\begin{lemma}\label{lem:giant:peeling}
  For any $\eps\in(0,1/2)$
  and $c\in(0,1-2\eps)$
  there exists $k_0=k_0(\eps,c)>0$ such that
  for every $k\ge k_0$,
  $G\sim G(n,k/n)$
  satisfies the following property \whp{}:
  every vertex set $U$ with $|U|\ge (1-\eps)n$
  contains a subset $U'$ with $|U'|\ge (1-2\eps)n$
  such that $\delta(G[U'])\ge ck$.
\end{lemma}

\begin{proof}
  We follow~\cite{AHHK22+}.
  Write $b=1-2\eps$.
  By Chernoff bounds (\cref{thm:chernoff})
  and the union bound,
  the probability that there exist disjoint vertex sets $A,B$
  with $|A|=\eps n$, $|B|\ge bn$ and
  $e(A,B)<\eps ckn$
  is at most
  \[
    2^n\cdot 2^n
    \cdot \pr(\bin(\eps bn^2,k/n) \le \eps ckn)
    = 4^n e^{-\frac{1}{2}(b-c)^2\eps kn}=o(1),
  \]
  for large enough $k$.
  We thus assume from now on that for every disjoint
  $A,B$ with $|A|=\eps n$ and $|B|\ge bn$
  we have that $e(A,B)\ge \eps ckn$.
  Fix $U$ with $|U|\ge (1-\eps)n$.
  We build a sequence $U_0,U_1,\ldots,U_m$  of subsets of $U$ in the following way: set $U_0=U$.
  For every $i\geq 0$, if $U_i$ has a vertex with degree smaller than $ck$ in $G[U_i]$,
  we denote that vertex by $u_{i+1}$, and set $U_{i+1}=U_i\sm\{u_{i+1}\}$;
  otherwise, set $m=i$. Let $U'=U_m$.
  Suppose that $m>\eps n$,
  and let $A=\{u_1,\dots,u_{\eps n}\}$.
  Let $B=U\sm A$, and note that $|B|\ge bn$.
  Note also that every $u\in A$ has less than $ck$ neighbours in $B$.
  Therefore,
  \begin{equation*}
    \eps ckn \le e(A,B) < |A|ck = \eps ckn,
  \end{equation*}
  a contradiction.
  Thus, $m\le \eps n$,
  and $|U'|\ge|U|-\eps n\ge (1-2\eps)n$.
  The statement follows since $\delta(G[U'])\ge ck$.
\end{proof}

\begin{proof}[Proof of \cref{lem:giant:balanced}]
  Let $B$ be the set of vertices of degree
  larger than $4k/\eps$,
  and let $m=|E(G)|$.
  Then, deterministically,
  $|B|\le 2m/(4k/\eps)=\eps m/(2k)$.
  Note that $m$ is distributed as a binomial random
  variable with $\binom{n}{2}$ attempts
  and success probability $k/n$;
  thus, by Chernoff bounds (\cref{thm:chernoff}),
  $\pr(m\ge kn) \le \exp(-0.6kn)=o(1)$.
  Thus, \whp{}, $|B|\le \eps n/2$.
  Let $U=V(G)\sm B$, so $|U|\ge (1-\eps/2)n$.
  By \cref{lem:giant:peeling}, for large enough $k$,
  there exists \whp{}
  $U'\subseteq U$ with $|U'|\ge (1-\eps)n$
  satisfying $\delta(G[U'])\ge c k$.
  Since $U'\subseteq U=V(G)\smallsetminus B$, we also have $\Delta(G[U'])\leq 4k/\eps$.
\end{proof}

Using the fact that $G(n,p)$ is typically $(p,\Theta(\sqrt{np}))$-jumbled
(see~\cite{KS06}),
we can readily infer from \cref{prop:jumbled:is:rigid} that for any $\eps\in(0,1)$,
$G(n,\Omega_\eps(d^2/n))$ includes a $d$-rigid component comprising at least $(1-\eps)n$ vertices.
To further refine the dependence of $p$ on $d$, we apply the same methodology used in previous sections.
This involves showing that the (nearly spanning) induced subgraph found in \cref{lem:giant:balanced}
is a ``sparse connector''.

\begin{proof}[Proof of \cref{thm:rigid_component}]
  Let $G\sim G(k/n)$, for $k$ to be determined later.
  Let $\eps\in(0,1)$ and set $c=(1-\eps)/2$.
  By \cref{lem:giant:balanced},
  there exists $k_0$ such that $k\ge k_0$ implies
  the existence of a set $U$ of size at least $(1-\eps)n$
  with $ck\le\delta(G[U])\le\Delta(G[U])\le 4k/\eps$.
  Write $G'=G[U]$ and $\delta=\delta(G')$.
  Let $\Gamma=4/(c\eps)$,
  and let $C=C(\Gamma)$ be the constant from \cref{prop:sparseconnector:is:rigid}.
  Let $x=3\log(np)/p$ and $y=5\log(np)$.
  By \cref{lem:gnp:sparse},
  there exists $k_1$ such that $k\ge k_1$ implies that
  $G$ (and hence $G'$) is \whp{} $(x,y)$-sparse.
  By \cref{lem:gnp:connect}, for $k\ge 4$,
  $G$ (and hence $G'$) is \whp{} a $(2x/3)$-connector.
  Take $C$ to be sufficiently large so that $k:=Cc^{-1}d\log{d}\ge\max\{k_0,k_1,4\}$
  and that $C\ge 35\log{k}/\log{d}$,
  and set $p=k/n$.
  Note that $\delta\ge Cd\log{d}$
  and that $5 \log(np)=5\log{k}\le C\log{d}/7 \leq\delta/(7d)$, so
  $G'$ is \whp{} $(x,\delta/(7d))$-sparse.
  Thus,
  by \cref{prop:sparseconnector:is:rigid}, $G'$ \whp{} admits a strong $d$-rigid partition,
  and is therefore (by \cref{cor:strong:rigid:partition}) $d$-rigid.
\end{proof}

\section{Dense graphs}\label{sec:dirac}
As we mentioned in the introduction,
\cref{thm:dirac} follows immediately from the following more general statement.
\begin{theorem}\label{thm:common_neighbours}
  Let $\ell=\ell(n)$ satisfy $3\log{n}\leq \ell< n$.
  Let $G=(V,E)$ be an $n$-vertex graph in which
  every pair of vertices has at least $\ell$ common neighbours.
  Then,
  $G$ admits a strong $d$-rigid partition
  for $d=  \ell/(3\log{n})$.
\end{theorem}
For the proof of \cref{thm:common_neighbours}, we will need the following simple lemma.

\begin{lemma}\label{lemma:many_neighbours}
Let $G=(V,E)$ be an $n$-vertex graph in which every pair of vertices has at least $\ell$ common neighbours. Let $U$ be a uniformly random subset of $V$ of size $t$. Then, with probability at least $1-n^{2}e^{-\ell t/n}$, every pair of vertices in $V$ has a common neighbour in $U$.
\end{lemma}
\begin{proof}
  Fix a pair of distinct vertices $u,v\in V$. Let $p$ be the probability that $U$ is disjoint from the set of common neighbours of $u,v$. Then,
  \[
      p\leq \frac{\binom{n-\ell}{t}}{\binom{n}{t}}=\prod_{i=1}^{\ell}\left(1-\frac{t}{n-\ell+i}\right) \leq \left(1-\frac{t}{n}\right)^{\ell}\leq e^{-\ell t/n}.
  \]
  The result follows by the union bound.
\end{proof}

\begin{proof}[Proof of \cref{thm:common_neighbours}]

Let $\ell=\ell(n)$ and let $d=\ell/(3\log{n})$.
  Partition our vertex set $V$ at random into sets $V_1,\ldots,V_d$, each of size $n/d= 3n\log{n}/\ell$. 
  By Lemma \ref{lemma:many_neighbours}, for every $1\leq i\leq d$,
  the probability that there is a pair of vertices in $V$ without a common neighbour in $V_i$ is at most $n^2 e^{-3\log{n}}=1/n$.
  By the union bound,
  the probability that there exist $i\in[d]$ and a pair of vertices in $V$
  without a common neighbour in $V_i$ is at most $d/n <1$. 
  Therefore, there exists a partition $V_1,\ldots,V_d$ such that every pair of vertices in $V$ has a common neighbour in $V_i$ for every $1\leq i\leq d$.

  We are left to show that $V_1,\ldots,V_d$ is a strong $d$-rigid partition of $G$.
  Indeed, for all $1\leq i\leq d$, $G[V_i]$ is connected,
  since every pair of distinct vertices $u,v\in V_i$ has a common neighbour in $V_i$.
  Similarly, let $i\neq j$, and let $u,v\in V_i\cup V_j$. If $u,v\in V_i$,
  then they have a common neighbour in $V_j$, so there is a path between them in $G[V_i,V_j]$.
  If $u\in V_i$ and $v\in V_j$, we divide into two cases: if $u$ and $v$ are adjacent, then we are done.
  Otherwise,  let $w$ be a neighbour of $u$ in $V_j$, and let $z$ be a common neighbour of $v$ and $w$ in $V_i$.
  Then $(u,w,z,v)$ is a path from $u$ to $v$ in $G[V_i,V_j]$. So, $G[V_i,V_j]$ is connected for all $i\neq j$.
  Therefore, $V_1,\ldots,V_d$ is a strong $d$-rigid partition of $G$.
\end{proof}

\begin{proof}[Proof of \cref{thm:dirac}]
  Let $u,v\in V$. Then, the number of common neighbours of $u$ and $v$ is at least
  $\deg(u)+\deg(v)-n\geq 2\ell$.
  Therefore, by \cref{thm:common_neighbours},
  $G$ admits a strong $d$-rigid partition for $d= 2\ell/(3\log{n})$.
\end{proof}

For the proof of \cref{thm:dirac:small},
we use the following well-known result concerning the existence of $k$-connected subgraphs.
Denote by $\bar{d}(G)$ the average degree of $G$.
\begin{theorem}[Mader~\cite{Mad72} (see also~\cite{Die})]\label{thm:mader}
  For every $k\ge 1$,
  if $\bar{d}(G)\ge 4k$,
  then $G$ contains a $(k+1)$-connected subgraph $H$
  with $\bar{d}(H)>\bar{d}(G)-2k$.
\end{theorem}

\begin{proof}[Proof of \cref{thm:dirac:small}]
  Let $C$ be the constant from \cref{thm:connectivity},
  take $c=\min\{\sqrt{1/8C},1/9\}$,
  and let $k=C(d+1)^2{\log^2}{n}$.
  Thus, $k\le n/8$.
  Since $\bar{d}(G)\ge\delta(G)\ge n/2+d\ge 4k$,
  by \cref{thm:mader}, $G$ contains a subgraph $H$ which is $k$-connected
  and has $\bar{d}(H)\ge n/2+d-2k\ge n/4$.
  Thus, by \cref{thm:connectivity}, $H$ is $(d+1)$-rigid and $|V(H)|\ge n/4$.
  Let $V_1$ be a maximum-sized vertex set for which $G[V_1]$ is $(d+1)$-rigid.
  If $V_1=V$, we are done, so assume otherwise.
  Write $n_1=|V_1|\ge n/4$.
  Let $V_2=V\sm V_1$ and write $G_2=G[V_2]$ and $n_2=|V_2|\le 3n/4$.
  Observe that every $v\in V_2$ has at most $d$ neighbours in $V_1$
  (as otherwise $V_1\cup\{v\}$ would have been a larger $(d+1)$-rigid component).
  Thus, $\delta(G_2)\ge\delta(G)-d\ge n/2$.
  In particular, $n_2 > n/2$, and therefore $n_2>n_1$.
  Set $\ell=\frac{3}{2}(d+1)\log{n_2}$,
  and note that $\ell\le \sqrt{n_2}/6\le n_2/6$.
  Thus, $\delta(G_2)\ge n/2\ge 2n_2/3=\left(\frac{1}{2}+\frac{1}{6}\right)n_2\ge\frac{1}{2}n_2+\ell$.
  By \cref{thm:dirac},
  $G_2$ is $(d+1)$-rigid, contradicting the maximality of $V_1$;
  thus $V_1=V$.
\end{proof}

Let us remark that applying Vill\'anyi's recent result~\cite{villanyi2023+every}*{Theorem 1.1}
instead of \cref{thm:connectivity}
allows for the condition on $d$ in \cref{thm:dirac:small} to be extended to $d=O(\sqrt{n})$.

\subsection*{Acknowledgements}
The second author thanks Eran Nevo, Yuval Peled and Orit Raz for many useful discussions.

\bibliography{library}

\begin{bibdiv}
\begin{biblist}

\bib{AZ}{book}{
      author={Aigner, Martin},
      author={Ziegler, G\"{u}nter~M.},
       title={Proofs from {T}he {B}ook},
     edition={Sixth Edition},
   publisher={Springer, Berlin},
        date={2018},
        ISBN={978-3-662-57264-1; 978-3-662-57265-8},
         url={https://doi.org/10.1007/978-3-662-57265-8},
      review={\MR{3823190}},
}

\bib{AHHK22+}{article}{
      author={{Aigner-Horev}, Elad},
      author={{Hefetz}, Dan},
      author={{Krivelevich}, Michael},
       title={{Minors, connectivity, and diameter in randomly perturbed sparse
  graphs}},
        date={2022-12},
     journal={arXiv e-prints},
      eprint={2212.07192},
}

\bib{AC88}{inproceedings}{
      author={Alon, Noga},
      author={Chung, Fan R.~K.},
       title={Explicit construction of linear sized tolerant networks},
        date={1988},
   booktitle={Proceedings of the {F}irst {J}apan {C}onference on {G}raph
  {T}heory and {A}pplications ({H}akone, 1986)},
      volume={72},
       pages={15\ndash 19},
         url={https://doi.org/10.1016/0012-365X(88)90189-6},
      review={\MR{975519}},
}

\bib{AS}{book}{
      author={Alon, Noga},
      author={Spencer, Joel~H.},
       title={The probabilistic method},
     edition={Fourth Edition},
      series={Wiley Series in Discrete Mathematics and Optimization},
   publisher={John Wiley \& Sons, Inc., Hoboken, NJ},
        date={2016},
        ISBN={978-1-119-06195-3},
      review={\MR{3524748}},
}

\bib{ALWZ21}{article}{
      author={Alweiss, Ryan},
      author={Lovett, Shachar},
      author={Wu, Kewen},
      author={Zhang, Jiapeng},
       title={Improved bounds for the sunflower lemma},
        date={2021},
        ISSN={0003-486X,1939-8980},
     journal={Annals of Mathematics. Second Series},
      volume={194},
      number={3},
       pages={795\ndash 815},
         url={https://doi.org/10.4007/annals.2021.194.3.5},
      review={\MR{4334977}},
}

\bib{AR78}{article}{
      author={Asimow, Leonard},
      author={Roth, Ben},
       title={The rigidity of graphs},
        date={1978},
        ISSN={0002-9947,1088-6850},
     journal={Transactions of the American Mathematical Society},
      volume={245},
       pages={279\ndash 289},
         url={https://doi.org/10.2307/1998867},
      review={\MR{511410}},
}

\bib{AR79}{article}{
      author={Asimow, Leonard},
      author={Roth, Ben},
       title={The rigidity of graphs.~{II}},
        date={1979},
        ISSN={0022-247X},
     journal={Journal of Mathematical Analysis and Applications},
      volume={68},
      number={1},
       pages={171\ndash 190},
         url={https://doi.org/10.1016/0022-247X(79)90108-2},
      review={\MR{531431}},
}

\bib{BR80}{article}{
      author={Bolker, Ethan~D.},
      author={Roth, Ben},
       title={When is a bipartite graph a rigid framework?},
        date={1980},
        ISSN={0030-8730,1945-5844},
     journal={Pacific Journal of Mathematics},
      volume={90},
      number={1},
       pages={27\ndash 44},
         url={http://projecteuclid.org/euclid.pjm/1102779115},
      review={\MR{599317}},
}

\bib{BCE80}{article}{
      author={Bollob\'{a}s, B\'ela},
      author={Catlin, Paul~A.},
      author={Erd\H{o}s, Paul},
       title={Hadwiger's conjecture is true for almost every graph},
        date={1980},
        ISSN={0195-6698,1095-9971},
     journal={European Journal of Combinatorics},
      volume={1},
      number={3},
       pages={195\ndash 199},
         url={https://doi.org/10.1016/S0195-6698(80)80001-1},
      review={\MR{593989}},
}

\bib{BT87}{article}{
      author={Bollob\'{a}s, B\'{e}la},
      author={Thomason, Andrew},
       title={Threshold functions},
        date={1987},
        ISSN={0209-9683},
     journal={Combinatorica},
      volume={7},
      number={1},
       pages={35\ndash 38},
         url={https://doi.org/10.1007/BF02579198},
      review={\MR{905149}},
}

\bib{BH}{book}{
      author={Brouwer, Andries~E.},
      author={Haemers, Willem~H.},
       title={Spectra of graphs},
      series={Universitext},
   publisher={Springer, New York},
        date={2012},
        ISBN={978-1-4614-1938-9},
         url={https://doi-org.cmu.idm.oclc.org/10.1007/978-1-4614-1939-6},
      review={\MR{2882891}},
}

\bib{CGGHK17}{article}{
      author={Censor-Hillel, Keren},
      author={Ghaffari, Mohsen},
      author={Giakkoupis, George},
      author={Haeupler, Bernhard},
      author={Kuhn, Fabian},
       title={Tight bounds on vertex connectivity under sampling},
        date={2017},
        ISSN={1549-6325,1549-6333},
     journal={ACM Transactions on Algorithms},
      volume={13},
      number={2},
       pages={Art. 19, 26},
         url={https://doi.org/10.1145/3086465},
      review={\MR{3659241}},
}

\bib{CGK14}{inproceedings}{
      author={Censor-Hillel, Keren},
      author={Ghaffari, Mohsen},
      author={Kuhn, Fabian},
       title={A new perspective on vertex connectivity},
        date={2014},
   booktitle={Proceedings of the {T}wenty-{F}ifth {A}nnual {ACM}-{SIAM}
  {S}ymposium on {D}iscrete {A}lgorithms},
   publisher={ACM, New York},
       pages={546\ndash 561},
         url={https://doi.org/10.1137/1.9781611973402.41},
      review={\MR{3376401}},
}

\bib{CN23}{article}{
      author={Chudnovsky, Maria},
      author={Nevo, Eran},
       title={Stable sets in flag spheres},
        date={2023},
        ISSN={0195-6698,1095-9971},
     journal={European Journal of Combinatorics},
      volume={110},
       pages={Paper No. 103699, 9},
         url={https://doi.org/10.1016/j.ejc.2023.103699},
      review={\MR{4549512}},
}

\bib{CDG21}{article}{
      author={Cioab\u{a}, Sebastian~M.},
      author={Dewar, Sean},
      author={Gu, Xiaofeng},
       title={Spectral conditions for graph rigidity in the {E}uclidean plane},
        date={2021},
        ISSN={0012-365X,1872-681X},
     journal={Discrete Mathematics},
      volume={344},
      number={10},
       pages={Paper No. 112527, 9},
         url={https://doi.org/10.1016/j.disc.2021.112527},
      review={\MR{4287763}},
}

\bib{clinch2022abstract}{article}{
      author={Clinch, Katie},
      author={Jackson, Bill},
      author={Tanigawa, Shin-ichi},
       title={Abstract 3-rigidity and bivariate {$C^1_2$}-splines {II}:
  {C}ombinatorial characterization},
        date={2022},
        ISSN={2397-3129},
     journal={Discrete Analysis},
       pages={Paper No. 3, 32},
         url={https://doi.org/10.19086/da},
      review={\MR{4450654}},
}

\bib{ConPG}{misc}{
      author={Conlon, David},
       title={Pseudorandom graphs},
         url={http://www.its.caltech.edu/~dconlon/PSlectures.pdf},
        note={Notes for a summer school},
}

\bib{Con93}{incollection}{
      author={Connelly, Robert},
       title={Rigidity},
        date={1993},
   booktitle={Handbook of convex geometry, {V}ol. {A}, {B}},
   publisher={North-Holland, Amsterdam},
       pages={223\ndash 271},
      review={\MR{1242981}},
}

\bib{Cra90}{techreport}{
      author={Crapo, Henry},
       title={{On the generic rigidity of plane frameworks}},
        type={Research Report},
 institution={{INRIA}},
        date={1990},
      number={RR-1278},
         url={https://inria.hal.science/inria-00075281},
        note={Projet ICSLA},
}

\bib{DM09}{article}{
      author={Devroye, Luc},
      author={Malalla, Ebrahim},
       title={On the {$k$}-orientability of random graphs},
        date={2009},
        ISSN={0012-365X,1872-681X},
     journal={Discrete Mathematics},
      volume={309},
      number={6},
       pages={1476\ndash 1490},
         url={https://doi.org/10.1016/j.disc.2008.02.023},
      review={\MR{2510554}},
}

\bib{Die}{book}{
      author={Diestel, Reinhard},
       title={Graph theory},
     edition={Fifth},
      series={Graduate Texts in Mathematics},
   publisher={Springer, Berlin},
        date={2018},
      volume={173},
        ISBN={978-3-662-57560-4; 978-3-662-53621-6},
      review={\MR{3822066}},
}

\bib{ER61}{article}{
      author={Erd\H{o}s, Paul},
      author={R\'{e}nyi, Alfr\'ed},
       title={On the strength of connectedness of a random graph},
        date={1961},
        ISSN={0001-5954},
     journal={Acta Mathematica. Academiae Scientiarum Hungaricae},
      volume={12},
       pages={261\ndash 267},
         url={https://doi-org.cmu.idm.oclc.org/10.1007/BF02066689},
      review={\MR{130187}},
}

\bib{FHL23}{article}{
      author={Fan, Dandan},
      author={Huang, Xueyi},
      author={Lin, Huiqiu},
       title={Spectral radius conditions for the rigidity of graphs},
        date={2023},
        ISSN={1077-8926},
     journal={Electronic Journal of Combinatorics},
      volume={30},
      number={2},
       pages={Paper No. 2.14, 14},
         url={https://doi.org/10.37236/11308},
      review={\MR{4578043}},
}

\bib{FR07}{inproceedings}{
      author={Fernholz, Daniel},
      author={Ramachandran, Vijaya},
       title={The {$k$}-orientability thresholds for {$G_{n,p}$}},
        date={2007},
   booktitle={Proceedings of the {E}ighteenth {A}nnual {ACM}-{SIAM} {S}ymposium
  on {D}iscrete {A}lgorithms},
   publisher={ACM, New York},
       pages={459\ndash 468},
      review={\MR{2482872}},
}

\bib{Fie73}{article}{
      author={Fiedler, Miroslav},
       title={Algebraic connectivity of graphs},
        date={1973},
        ISSN={0011-4642},
     journal={Czechoslovak Mathematical Journal},
      volume={23(98)},
       pages={298\ndash 305},
      review={\MR{318007}},
}

\bib{Fri08}{article}{
      author={Friedman, Joel},
       title={A proof of {A}lon's second eigenvalue conjecture and related
  problems},
        date={2008},
        ISSN={0065-9266},
     journal={Memoirs of the American Mathematical Society},
      volume={195},
      number={910},
       pages={viii+100},
         url={https://doi-org.cmu.idm.oclc.org/10.1090/memo/0910},
      review={\MR{2437174}},
}

\bib{FKS89}{inproceedings}{
      author={Friedman, Joel},
      author={Kahn, Jeff},
      author={Szemer\'{e}di, Endre},
       title={On the second eigenvalue of random regular graphs},
        date={1989},
   booktitle={Proceedings of the {T}wenty-{F}irst {A}nnual {ACM} {S}ymposium on
  {T}heory of {C}omputing},
      series={STOC '89},
   publisher={Association for Computing Machinery},
     address={New York, NY, USA},
       pages={587–598},
         url={https://doi.org/10.1145/73007.73063},
}

\bib{FK}{book}{
      author={Frieze, Alan},
      author={Karo\'{n}ski, Micha\l},
       title={Introduction to random graphs},
   publisher={Cambridge University Press, Cambridge},
        date={2016},
        ISBN={978-1-107-11850-8},
         url={https://doi-org.cmu.idm.oclc.org/10.1017/CBO9781316339831},
      review={\MR{3675279}},
}

\bib{GLS15}{article}{
      author={Glebov, Roman},
      author={Liebenau, Anita},
      author={Szab\'{o}, Tibor},
       title={On the concentration of the domination number of the random
  graph},
        date={2015},
        ISSN={0895-4801},
     journal={SIAM Journal on Discrete Mathematics},
      volume={29},
      number={3},
       pages={1186\ndash 1206},
         url={https://doi.org/10.1137/12090054X},
      review={\MR{3369991}},
}

\bib{Gra91}{article}{
      author={Graver, Jack~E.},
       title={Rigidity matroids},
        date={1991},
        ISSN={0895-4801},
     journal={SIAM Journal on Discrete Mathematics},
      volume={4},
      number={3},
       pages={355\ndash 368},
         url={https://doi.org/10.1137/0404032},
      review={\MR{1105942}},
}

\bib{GSS}{book}{
      author={Graver, Jack~E.},
      author={Servatius, Brigitte},
      author={Servatius, Herman},
       title={Combinatorial rigidity},
      series={Graduate Studies in Mathematics},
   publisher={American Mathematical Society, Providence, RI},
        date={1993},
      volume={2},
        ISBN={0-8218-3801-6},
         url={https://doi.org/10.1090/gsm/002},
      review={\MR{1251062}},
}

\bib{HLW06}{article}{
      author={Hoory, Shlomo},
      author={Linial, Nathan},
      author={Wigderson, Avi},
       title={Expander graphs and their applications},
        date={2006},
        ISSN={0273-0979},
     journal={American Mathematical Society. Bulletin. New Series},
      volume={43},
      number={4},
       pages={439\ndash 561},
         url={https://doi-org.cmu.idm.oclc.org/10.1090/S0273-0979-06-01126-8},
      review={\MR{2247919}},
}

\bib{JJ05}{article}{
      author={Jackson, Bill},
      author={Jord\'{a}n, Tibor},
       title={The {$d$}-dimensional rigidity matroid of sparse graphs},
        date={2005},
        ISSN={0095-8956},
     journal={Journal of Combinatorial Theory. Series B},
      volume={95},
      number={1},
       pages={118\ndash 133},
         url={https://doi-org.cmu.idm.oclc.org/10.1016/j.jctb.2005.03.004},
      review={\MR{2156343}},
}

\bib{JSS07}{article}{
      author={Jackson, Bill},
      author={Servatius, Brigitte},
      author={Servatius, Herman},
       title={The 2-dimensional rigidity of certain families of graphs},
        date={2007},
        ISSN={0364-9024,1097-0118},
     journal={Journal of Graph Theory},
      volume={54},
      number={2},
       pages={154\ndash 166},
         url={https://doi.org/10.1002/jgt.20196},
      review={\MR{2285456}},
}

\bib{JLR}{book}{
      author={Janson, Svante},
      author={{\L}uczak, Tomasz},
      author={Ruci\'nski, Andrzej},
       title={Random graphs},
      series={Wiley-Interscience Series in Discrete Mathematics and
  Optimization},
   publisher={Wiley-Interscience, New York},
        date={2000},
        ISBN={0-471-17541-2},
         url={https://doi.org/10.1002/9781118032718},
      review={\MR{1782847}},
}

\bib{Jor12}{article}{
      author={Jord\'{a}n, Tibor},
       title={Highly connected molecular graphs are rigid in three dimensions},
        date={2012},
        ISSN={0020-0190,1872-6119},
     journal={Information Processing Letters},
      volume={112},
      number={8-9},
       pages={356\ndash 359},
         url={https://doi.org/10.1016/j.ipl.2012.01.013},
      review={\MR{2881297}},
}

\bib{Jor22}{article}{
      author={Jord{\'a}n, Tibor},
       title={The globally rigid complete bipartite graphs},
        date={2022},
     journal={EGRES Quick-Proofs, QP-2022-02},
        note={available at \url{https://egres.elte.hu/qp/egresqp-22-02.pdf}},
}

\bib{JT22pow}{article}{
      author={Jord\'{a}n, Tibor},
      author={Tanigawa, Shin-ichi},
       title={Globally rigid powers of graphs},
        date={2022},
        ISSN={0095-8956,1096-0902},
     journal={Journal of Combinatorial Theory. Series B},
      volume={155},
       pages={111\ndash 140},
         url={https://doi.org/10.1016/j.jctb.2022.02.004},
      review={\MR{4387282}},
}

\bib{JT22}{article}{
      author={Jord\'{a}n, Tibor},
      author={Tanigawa, Shin-ichi},
       title={Rigidity of random subgraphs and eigenvalues of stiffness
  matrices},
        date={2022},
        ISSN={0895-4801,1095-7146},
     journal={SIAM Journal on Discrete Mathematics},
      volume={36},
      number={3},
       pages={2367\ndash 2392},
         url={https://doi.org/10.1137/20M1349849},
      review={\MR{4488584}},
}

\bib{KMT11}{inproceedings}{
      author={Kasiviswanathan, Shiva~Prasad},
      author={Moore, Cristopher},
      author={Theran, Louis},
       title={The rigidity transition in random graphs},
        date={2011},
   booktitle={Proceedings of the {T}wenty-{S}econd {A}nnual {ACM}-{SIAM}
  {S}ymposium on {D}iscrete {A}lgorithms},
   publisher={SIAM, Philadelphia, PA},
       pages={1237\ndash 1252},
      review={\MR{2858396}},
}

\bib{KT13+}{article}{
      author={{Kir{\'a}ly}, Franz~J.},
      author={{Theran}, Louis},
       title={{Coherence and sufficient sampling densities for reconstruction
  in compressed sensing}},
        date={2013-02},
     journal={arXiv e-prints},
      eprint={1302.2767},
}

\bib{Kri16}{incollection}{
      author={Krivelevich, Michael},
       title={Long paths and hamiltonicity in random graphs},
        date={2016},
   booktitle={Random graphs, geometry and asymptotic structure},
      editor={Fountoulakis, Nikolaos},
      editor={Hefetz, Dan},
      series={London Mathematical Society Student Texts},
   publisher={Cambridge University Press},
       pages={4–27},
}

\bib{KS06}{incollection}{
      author={Krivelevich, Michael},
      author={Sudakov, Benny},
       title={Pseudo-random graphs},
        date={2006},
   booktitle={More sets, graphs and numbers},
      editor={Gy{\H{o}}ri, Ervin},
      editor={Katona, Gyula O.~H.},
      editor={Lov{\'a}sz, L{\'a}szl{\'o}},
      editor={Fleiner, Tam{\'a}s},
      series={Bolyai Soc. Math. Stud.},
      volume={15},
   publisher={Springer, Berlin},
       pages={199\ndash 262},
         url={https://doi-org.cmu.idm.oclc.org/10.1007/978-3-540-32439-3_10},
      review={\MR{2223394}},
}

\bib{KO03}{article}{
      author={K\"{u}hn, Daniela},
      author={Osthus, Deryk},
       title={Partitions of graphs with high minimum degree or connectivity},
        date={2003},
        ISSN={0095-8956,1096-0902},
     journal={Journal of Combinatorial Theory. Series B},
      volume={88},
      number={1},
       pages={29\ndash 43},
         url={https://doi.org/10.1016/S0095-8956(03)00028-5},
      review={\MR{1973257}},
}

\bib{Lam70}{article}{
      author={Laman, Gerard},
       title={On graphs and rigidity of plane skeletal structures},
        date={1970},
        ISSN={0022-0833,1573-2703},
     journal={Journal of Engineering Mathematics},
      volume={4},
       pages={331\ndash 340},
         url={https://doi.org/10.1007/BF01534980},
      review={\MR{269535}},
}

\bib{Lan53}{article}{
      author={Landau, Hyman~G.},
       title={On dominance relations and the structure of animal societies.
  {III}. {T}he condition for a score structure},
        date={1953},
        ISSN={0007-4985,1522-9602},
     journal={The Bulletin of Mathematical Biophysics},
      volume={15},
       pages={143\ndash 148},
         url={https://doi.org/10.1007/bf02476378},
      review={\MR{54933}},
}

\bib{LNPR22+}{article}{
      author={{Lew}, Alan},
      author={{Nevo}, Eran},
      author={{Peled}, Yuval},
      author={{Raz}, Orit~E.},
       title={{On the $d$-dimensional algebraic connectivity of graphs}},
        date={2022-05},
     journal={arXiv e-prints},
      eprint={2205.05530},
}

\bib{LNPR23+}{article}{
      author={{Lew}, Alan},
      author={{Nevo}, Eran},
      author={{Peled}, Yuval},
      author={{Raz}, Orit~E.},
       title={{Rigidity expander graphs}},
        date={2023-04},
     journal={arXiv e-prints},
      eprint={2304.01306},
}

\bib{LNPR23}{article}{
      author={Lew, Alan},
      author={Nevo, Eran},
      author={Peled, Yuval},
      author={Raz, Orit~E.},
       title={Sharp threshold for rigidity of random graphs},
        date={2023},
     journal={Bulletin of the London Mathematical Society},
      volume={55},
      number={1},
       pages={490\ndash 501},
  url={https://londmathsoc.onlinelibrary.wiley.com/doi/abs/10.1112/blms.12740},
}

\bib{LinPhD}{thesis}{
      author={Lindemann, Tim},
       title={Combinatorial aspects of spatial frameworks},
        type={Ph.D. Thesis},
organization={Universit{\"a}t Bremen},
        date={2022},
}

\bib{LY82}{article}{
      author={Lov\'{a}sz, L\'aszl\'o},
      author={Yemini, Yechiam},
       title={On generic rigidity in the plane},
        date={1982},
        ISSN={0196-5212},
     journal={Society for Industrial and Applied Mathematics. Journal on
  Algebraic and Discrete Methods},
      volume={3},
      number={1},
       pages={91\ndash 98},
         url={https://doi-org.cmu.idm.oclc.org/10.1137/0603009},
      review={\MR{644960}},
}

\bib{Mad72}{article}{
      author={Mader, Wolfgang},
       title={Existenz {$n$}-fach zusammenh\"{a}ngender {T}eilgraphen in
  {G}raphen gen\"{u}gend grosser {K}antendichte},
        date={1972},
        ISSN={0025-5858,1865-8784},
     journal={Abhandlungen aus dem Mathematischen Seminar der Universit\"{a}t
  Hamburg},
      volume={37},
       pages={86\ndash 97},
         url={https://doi.org/10.1007/BF02993903},
      review={\MR{306050}},
}

\bib{Nix21}{article}{
      author={Nixon, Anthony},
      author={Schulze, Bernd},
      author={Whiteley, Walter},
       title={Rigidity through a projective lens},
        date={2021},
        ISSN={2076-3417},
     journal={Applied Sciences},
      volume={11},
      number={24},
         url={https://www.mdpi.com/2076-3417/11/24/11946},
}

\bib{PG27}{article}{
      author={Pollaczek-Geiringer, Hilda},
       title={Über die gliederung ebener fachwerke},
        date={1927},
     journal={ZAMM --- Journal of Applied Mathematics and Mechanics /
  Zeitschrift für Angewandte Mathematik und Mechanik},
      volume={7},
      number={1},
       pages={58\ndash 72},
  url={https://onlinelibrary.wiley.com/doi/abs/10.1002/zamm.19270070107},
}

\bib{Ray84}{article}{
      author={Raymond, Jean-Luc},
       title={Generic rigidity of complete bipartite graphs in {${\bf R}^d$}},
        date={1984},
        ISSN={0226-9171},
     journal={Structural Topology},
      number={10},
       pages={57\ndash 62},
        note={Dual French-English text},
      review={\MR{768706}},
}

\bib{Sar23}{article}{
      author={Sarid, Amir},
       title={The spectral gap of random regular graphs},
        date={2023},
        ISSN={1042-9832,1098-2418},
     journal={Random Structures \& Algorithms},
      volume={63},
      number={2},
       pages={557\ndash 587},
      review={\MR{4622336}},
}

\bib{schulze2017rigidity}{incollection}{
      author={Schulze, Bernd},
      author={Whiteley, Walter},
       title={Rigidity and scene analysis},
        date={2017},
   booktitle={Handbook of discrete and computational geometry},
   publisher={Chapman and Hall/CRC},
       pages={1593\ndash 1632},
}

\bib{SS02}{incollection}{
      author={Servatius, Brigitte},
      author={Servatius, Herman},
       title={Generic and abstract rigidity},
        date={2002},
   booktitle={Rigidity theory and applications},
      editor={Thorpe, M.~F.},
      editor={Duxbury, P.~M.},
   publisher={Springer US},
     address={Boston, MA},
       pages={1\ndash 19},
         url={https://doi.org/10.1007/0-306-47089-6_1},
}

\bib{SJ17}{article}{
      author={Sidman, Jessica},
      author={St.~John, Audrey},
       title={The rigidity of frameworks: theory and applications},
        date={2017},
        ISSN={0002-9920,1088-9477},
     journal={Notices of the American Mathematical Society},
      volume={64},
      number={9},
       pages={973\ndash 978},
         url={https://doi.org/10.1090/noti1575},
      review={\MR{3699772}},
}

\bib{Tay93}{article}{
      author={Tay, Tiong-Seng},
       title={A new proof of {L}aman's theorem},
        date={1993},
        ISSN={0911-0119,1435-5914},
     journal={Graphs Combin.},
      volume={9},
      number={4},
       pages={365\ndash 370},
         url={https://doi.org/10.1007/BF02988323},
      review={\MR{1250514}},
}

\bib{Tho87a}{incollection}{
      author={Thomason, Andrew},
       title={Pseudorandom graphs},
        date={1987},
   booktitle={Random graphs '85},
      series={North-Holland Math. Stud.},
      volume={144},
   publisher={North-Holland, Amsterdam},
       pages={307\ndash 331},
      review={\MR{930498}},
}

\bib{Tho87b}{incollection}{
      author={Thomason, Andrew},
       title={Random graphs, strongly regular graphs and pseudorandom graphs},
        date={1987},
   booktitle={Surveys in combinatorics},
      series={London Math. Soc. Lecture Note Ser.},
      volume={123},
   publisher={Cambridge Univ. Press, Cambridge},
       pages={173\ndash 195},
      review={\MR{905280}},
}

\bib{villanyi2023+every}{article}{
      author={Vill\'anyi, Soma},
       title={{Every $d(d+1)$-connected graph is globally rigid in
  $\mathbb{R}^d$}},
        date={2023-12},
     journal={arXiv e-prints},
      eprint={2312.02028},
}

\bib{Whi84}{article}{
      author={Whiteley, Walter},
       title={Infinitesimal motions of a bipartite framework},
        date={1984},
        ISSN={0030-8730,1945-5844},
     journal={Pacific Journal of Mathematics},
      volume={110},
      number={1},
       pages={233\ndash 255},
         url={http://projecteuclid.org/euclid.pjm/1102711115},
      review={\MR{722753}},
}

\bib{whiteley1995some}{incollection}{
      author={Whiteley, Walter},
       title={Some matroids from discrete applied geometry},
        date={1996},
   booktitle={Matroid theory ({S}eattle, {WA}, 1995)},
      series={Contemp. Math.},
      volume={197},
   publisher={Amer. Math. Soc., Providence, RI},
       pages={171\ndash 311},
         url={https://doi.org/10.1090/conm/197/02540},
      review={\MR{1411692}},
}

\bib{Zas10}{incollection}{
      author={Zaslavsky, Thomas},
       title={Matrices in the theory of signed simple graphs},
        date={2010},
   booktitle={Advances in discrete mathematics and applications: {M}ysore,
  2008},
      series={Ramanujan Math. Soc. Lect. Notes Ser.},
      volume={13},
   publisher={Ramanujan Math. Soc., Mysore},
       pages={207\ndash 229},
      review={\MR{2766941}},
}

\bib{Zel86}{article}{
      author={Zelinka, Bohdan},
       title={Connected domatic number of a graph},
        date={1986},
        ISSN={0025-5173},
     journal={Mathematica Slovaca},
      volume={36},
      number={4},
       pages={387\ndash 392},
      review={\MR{871778}},
}

\end{biblist}
\end{bibdiv}

\end{document}